\documentclass[11pt,reqno]{amsart}
\usepackage[margin=0.8in]{geometry}
\usepackage{amsmath,amssymb,amsthm,mathtools}
\numberwithin{equation}{section}
\usepackage{enumitem}
\usepackage[hidelinks]{hyperref}

\newtheorem{theorem}{Theorem}[section]
\newtheorem{lemma}[theorem]{Lemma}
\newtheorem{proposition}[theorem]{Proposition}

\theoremstyle{remark}
\newtheorem{remark}[theorem]{Remark}

\newcommand{\Irr}{\operatorname{Irr}}
\newcommand{\Hall}{\operatorname{Hall}}
\newcommand{\PSL}{\operatorname{PSL}}
\newcommand{\PGL}{\operatorname{PGL}}
\newcommand{\SL}{\operatorname{SL}}
\newcommand{\GL}{\operatorname{GL}}
\newcommand{\PSU}{\operatorname{PSU}}
\newcommand{\PGU}{\operatorname{PGU}}

\newcommand{\GU}{\operatorname{GU}}
\newcommand{\Syl}{\operatorname{Syl}}
\newcommand{\St}{\operatorname{St}}
\newcommand{\Sp}{\operatorname{Sp}}
\newcommand{\PSp}{\operatorname{PSp}}

\title{Characterizing normal Hall subgroups by character degrees, II}
\author{YONG YANG}
\address{Department of Mathematics, Texas State University, San Marcos, TX 78666, USA.}
\email{yang@txstate.edu}


\begin{document}
\begin{abstract}
We prove that a Hall $\pi$-subgroup $H$ of a finite group $G$ is normal if and only if every irreducible constituent of the permutation character $1_H^G$ has $\pi'$-degree.
\end{abstract}

\maketitle

\section{Introduction}

Let $G$ be a finite group and let $H$ be a Hall $\pi$-subgroup of $G$.  The question considered in \cite{LYZ} asks whether normality of $H$ can be characterized by the degrees of the irreducible constituents of $1_H^G$.  This is motivated by a theorem of Malle and Navarro \cite{MN12}: if $P\in\Syl_p(G)$, then $P$ is normal in $G$ if and only if every irreducible constituent of $1_P^G$ has degree prime to $p$.  Thus the Hall-subgroup question asks whether the same characterization holds when a single prime $p$ is replaced by a set of primes $\pi$.  For a nonabelian simple group, the relevant assertion is the following.

\begin{quote}
If $1<H<G$ and $H\in\Hall_\pi(G)$, then there exists $\chi\in\Irr(1_H^G)$ such that $\chi(1)$ is divisible by a prime in $\pi$.
\end{quote}

The restriction $H<G$ is necessary: if $H=G$, then $1_H^G=1_G$.  This is the simple-group assertion needed in the reduction of \cite{LYZ}.

For proper nontrivial Hall subgroups, the assertion is known for alternating groups, sporadic groups, and exceptional groups of Lie type; see \cite{LYZ}.  It therefore remains to prove the assertion for classical simple groups.  Throughout the paper, $p$ denotes the defining characteristic when $G$ is a group of Lie type, and $e(q,t)$ denotes the multiplicative order of $q$ modulo a prime $t\nmid q$.

Suppose first that $2\notin\pi$.  In nondefining characteristic, a $\pi$-defect-zero character is enough: if $\chi\in\Irr(G)$ has $\pi$-defect zero, then $\chi$ vanishes on every nonidentity $\pi$-element and
\[
\chi_H=\frac{\chi(1)}{|H|}\rho_H,
\]
so $[\chi_H,1_H]>0$.  Gross and Vdovin--Revin classify the odd-order Hall subgroups of the classical groups; we use the formulation in \cite{VR02}.  If $r=\min\pi$ and $\tau=\pi\setminus\{r\}$, every Hall $\pi$-subgroup contains a normal abelian Hall $\tau$-subgroup \cite[Theorem~1]{VR02}, and their classification restricts the cyclotomic parameters to finitely many possibilities.  Except for a few split cases, one can choose a regular semisimple label with $\pi'$-centralizer and hence a $\pi$-defect-zero semisimple character.  The split linear case is handled by the permutation action on $2$-spaces, while the unitary case uses the rank-three action on isotropic points.  In the symplectic and orthogonal groups the classification also provides suitable maximal tori; a Clifford-theoretic argument then passes from an inner-diagonal group to the simple derived subgroup.

If $p\in\pi$ and $2\notin\pi$, Revin's classification shows that a Hall $\pi$-subgroup is contained in a Borel subgroup: the other parabolic possibilities have even order.  The Steinberg character then has a nonzero fixed vector and degree divisible by $p$.  Now suppose that $2\in\pi$.  In defining characteristic we again use Revin's classification.  In nondefining characteristic the Revin--Vdovin classification places the Hall subgroup in a decomposition stabilizer or a torus normalizer, apart from three exceptional orthogonal embeddings.  The decomposition and torus-normalizer cases are handled by permutation characters, and the three orthogonal exceptions by rank-three actions together with an orbit-divisibility argument.

\section{Defect-zero characters}

\begin{lemma}\label{lem:defect}
Let $G$ be a finite group, $H\in\Hall_\pi(G)$, and $\chi\in\Irr(G)$.  If $\chi$ has $\pi$-defect zero, then $\chi\in\Irr(1_H^G)$.
\end{lemma}

\begin{proof}
Every nonidentity element of $H$ is a $\pi$-element.  A $\pi$-defect-zero character vanishes on such elements.  Hence
\[
\chi_H=\frac{\chi(1)}{|H|}\rho_H,
\]
and therefore
\[
[\chi_H,1_H]=\frac{\chi(1)}{|H|}>0.
\]
\end{proof}

\begin{lemma}\label{lem:two-extension}
Let $N\lhd X$ and suppose that $X/N$ is a $2$-group.  Let $\pi$ be a set of odd primes and let $H\in\Hall_\pi(N)$.  If $\chi\in\Irr(X)$ has $\pi$-defect zero, then some constituent $\psi\in\Irr(\chi_N)$ has $\pi$-defect zero in $N$ and satisfies
\[
[\psi_H,1_H]>0.
\]
\end{lemma}

\begin{proof}
Since $|X:N|$ is a power of $2$, we have $|X|_\pi=|N|_\pi=|H|$, and hence $H$ is also a Hall $\pi$-subgroup of $X$.  Lemma~\ref{lem:defect} gives
\[
[\chi_H,1_H]>0.
\]
By Clifford theory,
\[
\chi_N=e(\psi_1+\cdots+\psi_t),
\]
where the $\psi_i$ are $X$-conjugate.  Both $e$ and $t$ divide powers of $|X:N|$, and therefore are powers of $2$.  Hence
\[
\chi(1)_\pi=\psi_i(1)_\pi
\]
for every $i$.  Since $\chi$ has $\pi$-defect zero,
\[
\psi_i(1)_\pi=|X|_\pi=|N|_\pi,
\]
so every $\psi_i$ has $\pi$-defect zero in $N$.  Finally,
\[
0<[\chi_H,1_H]
=e\sum_{i=1}^t[(\psi_i)_H,1_H],
\]
and therefore at least one $\psi_i$ has a nonzero $H$-fixed vector.
\end{proof}

We use semisimple characters attached to regular semisimple labels.  Let $X$ be a finite group of Lie type in characteristic $p$, let $X^*$ be a dual group, and let $s\in X^*$ be regular semisimple.  The connected centralizer of $s$ in the ambient algebraic group is a maximal torus, while the full finite centralizer $C_{X^*}(s)$ may have a component group.  A semisimple character $\chi_s$ labelled by $s$ has degree
\[
\chi_s(1)=|X^*:C_{X^*}(s)|_{p'}.
\]
See \cite[Section~12.4]{DM20} for regular and semisimple characters.  Consequently, since $|X|=|X^*|$, if $p\notin\pi$ and $C_{X^*}(s)$ is a $\pi'$-group, then
\[
\chi_s(1)_\pi=|X|_\pi,
\]
so $\chi_s$ has $\pi$-defect zero.

We also use the compatibility of Deligne--Lusztig induction with central quotients; see Digne--Michel \cite[Propositions~11.3.10 and~11.4.9]{DM20}.  Hence, in type $A$, semisimple characters labelled in the projective dual group descend to the relevant simple central quotient.  The linear arguments therefore use regular semisimple labels in $\PGL_n(q)$, while the unitary arguments use labels in $\PGU_n(q)$.  For the two $\PSU_3(q)$ Singer-torus calculations, we work directly in the projective group, so no determinant-one representative is needed.

\begin{lemma}\label{lem:typeA-regular-tori}
Let $X^*$ be $\PGL_n(q)$ or $\PGU_n(q)$.  Each Singer torus used below, and each torus of type $(n-1,1)$ used below, contains an element whose centralizer in $X^*$ is precisely that torus.  More generally, let $T^*$ be a projective torus of type $(a,b)$ with $a\ne b$.  Then $T^*$ contains an element $s$ such that
\[
[C_{X^*}(s):T^*]\mid \gcd(a,b).
\]
\end{lemma}

\begin{proof}
Consider the usual field-extension model.  In the linear case a Singer factor of dimension $a$ is obtained from $\mathbb F_{q^a}^{\times}$ acting by multiplication.  Choose an element of degree $a$ over $\mathbb F_q$.  Its eigenvalues over an algebraic closure are its $q$-conjugates, and these are pairwise distinct.  For a torus of type $(n-1,1)$ choose such an element on the $(n-1)$-dimensional block and choose the eigenvalue on the $1$-dimensional block outside that Frobenius orbit.  Thus all eigenvalues are distinct.  The centralizer in $\GL_n(q)$ is the corresponding torus, and quotienting by scalars gives the asserted centralizer in $\PGL_n(q)$.

Now let $T^*$ have type $(a,b)$ with $a\ne b$.  Choose full-degree elements on both blocks.  Their Frobenius orbits are disjoint because they have different lengths, so their centralizer in the full linear group is the corresponding torus.  If an element centralizes the image of $s$ in the projective group, then it conjugates $s$ to $zs$ for some scalar $z$.  Multiplication by $z$ preserves each of the two Frobenius orbits, since the orbit lengths are different.  On an orbit of length $a$ it acts as a power of Frobenius, so the order of $z$ divides $a$; similarly it divides $b$.  Hence the possible scalar multipliers form a group of order dividing $\gcd(a,b)$, and therefore
\[
[C_{X^*}(s):T^*]\mid\gcd(a,b).
\]

For unitary groups, replace the Frobenius map $x\mapsto x^q$ by the unitary Frobenius $x\mapsto x^{-q}$.  The preceding orbit argument then proves the Singer and $(n-1,1)$ assertions and, for a torus of type $(a,b)$ with $a\ne b$, the displayed divisibility for the projective centralizer.
\end{proof}

\section{Linear groups}

Let
\[
G=\PSL_n(q),\qquad q=p^f,
\]
and suppose that $H\in\Hall_\pi(G)$ with $2,p\notin\pi$.  For an odd prime $t\ne p$, write
\[
e_t=e(q,t).
\]
The pairwise Hall classification for $\GL_n(q)$ gives the following four patterns for $t<s$ in $\pi$; see the proof of \cite[Theorem~4]{VR02} and Gross \cite{Gross95}:
\begin{enumerate}[label=\textup{(\roman*)}]
\item $e_t=e_s$;
\item $e_t=t-1$, $e_s=t$ and the corresponding floor equality holds;
\item $e_t=t-1$, $e_s=t$ and the adjacent floor equality holds;
\item $e_t=t-1$, $e_s=1$ and
\[
\left\lfloor\frac{n}{t-1}\right\rfloor=\left\lfloor\frac nt\right\rfloor.
\]
\end{enumerate}
Moreover the mixed cases can occur only when $t=r=\min\pi$.  Thus all primes in $\tau=\pi\setminus\{r\}$ have one common cyclotomic parameter.

\subsection{Equal multiplicative orders}

\begin{proposition}\label{prop:linear-equal}
Assume that $e(q,t)=d$ for every $t\in\pi$.
If $d>1$, then $G$ has a $\pi$-defect-zero character.  If $d=1$ and no prime in $\pi$ divides $n$, the same conclusion holds.
\end{proposition}

\begin{proof}
Suppose first that $d>1$, and put $G^*=\PGL_n(q)$.  If $d\nmid n$, let $T^*$ be a projective Singer torus of $G^*$.  Then
\[
 |T^*|=\frac{q^n-1}{q-1}.
\]
No prime $t\in\pi$ divides this number, since $e(q,t)=d\nmid n$ and $d>1$.  By Lemma~\ref{lem:typeA-regular-tori}, choose $s\in T^*$ with $C_{G^*}(s)=T^*$, so the associated semisimple character has $\pi$-defect zero.

If $d\mid n$, let $T^*$ be a projective torus of type $(n-1,1)$.  Its order is
\[
 |T^*|=q^{n-1}-1.
\]
Since $d\nmid n-1$, this is a $\pi'$-number.  By Lemma~\ref{lem:typeA-regular-tori}, there is $s\in T^*$ with $C_{G^*}(s)=T^*$.  The associated semisimple character has $\pi$-defect zero.

Now assume $d=1$.  Then every $t\in\pi$ divides $q-1$.  The projective Singer torus has order
\[
\frac{q^n-1}{q-1}.
\]
For every $t\in\pi$, LTE gives
\[
v_t\!\left(\frac{q^n-1}{q-1}\right)=v_t(n).
\]
Hence, if no prime in $\pi$ divides $n$, this torus is a $\pi'$-group.  By Lemma~\ref{lem:typeA-regular-tori}, this torus contains an element whose projective centralizer is the torus itself, and the associated semisimple character has $\pi$-defect zero.
\end{proof}

It remains to consider
\[
e(q,t)=1\quad(t\in\pi),\qquad r\mid n.
\]
In this case a $\pi$-defect-zero character need not exist, so we use a constituent of a permutation character.

Let $\Omega_j$ be the set of $j$-dimensional subspaces of the natural module, and let $P_j$ be the stabilizer of a member of $\Omega_j$.  The permutation character on $2$-spaces has the form
\[
1_{P_2}^{\GL_n(q)}=1+\chi^{(n-1,1)}+\chi^{(n-2,2)}.
\]
The character $\chi^{(n-2,2)}$ is unipotent and
\begin{equation}\label{eq:two-row-degree}
\chi^{(n-2,2)}(1)={n\brack2}_q-{n\brack1}_q.
\end{equation}

\begin{lemma}\label{lem:linear-mon}
Let $M$ be the full monomial subgroup of $\GL_n(q)$.  If $n\ge4$, then
\[
[\chi^{(n-2,2)},1_M^{\GL_n(q)}]>0.
\]
\end{lemma}

\begin{proof}
Choose the standard basis $e_1,\ldots,e_n$ with respect to which $M$ is monomial.  Let $a_j$ be the number of $M$-orbits on $j$-spaces.  The orbits on projective points are classified by support size, so $a_1=n$.

For a $2$-space $U$, let $s(U)$ be its support size and let $c(U)$ be the number of coordinate lines contained in $U$.  For $3\le s\le n$, the spaces
\[
U_s=\langle e_1,e_2+\cdots+e_s\rangle
\]
have $(s(U_s),c(U_s))=(s,1)$.  The coordinate plane has invariant $(2,2)$.  In addition,
\[
V_3=\langle e_1+e_2,e_2+e_3\rangle,
\qquad
V_4=\langle e_1+e_2,e_2+e_3+e_4\rangle
\]
have invariants $(3,0)$ and $(4,0)$, respectively.  Hence $a_2\ge n+1$.  Since
\[
a_2-a_1=[\chi^{(n-2,2)},1_M^{\GL_n(q)}],
\]
the assertion follows.
\end{proof}

\begin{proposition}\label{prop:linear-split}
Assume that $e(q,t)=1$ for every $t\in\pi$ and that $r\mid n$.  Then there exists
\[
\chi\in\Irr(1_H^G)
\]
with $r\mid\chi(1)$.
\end{proposition}

\begin{proof}
If $n=3$, then $r=3$.  Since every prime $t\in\pi$ satisfies $e(q,t)=1$, all Hall primes divide $q-1$.  Let $T^*\leq G^*=\PGL_3(q)$ be a projective Singer torus, so that
\[
|T^*|=q^2+q+1.
\]
By Lemma~\ref{lem:typeA-regular-tori}, choose $s\in T^*$ with $C_{G^*}(s)=T^*$.  The corresponding semisimple character on the simple quotient has degree
\[
\chi_s(1)=(q-1)^2(q+1),
\]
by the central-quotient compatibility recalled above.  On the other hand,
\[
|\PSL_3(q)|_{p'}=
\frac{(q-1)^2(q+1)(q^2+q+1)}{(3,q-1)}.
\]
Now $\gcd(q-1,q^2+q+1)=\gcd(q-1,3)$.  Since $3\mid q-1$, LTE gives $v_3(q^2+q+1)=1$, so the factor $3$ in $q^2+q+1$ is exactly cancelled by the denominator $(3,q-1)$.  No other Hall prime divides $q^2+q+1$.  Hence $\chi_s$ contains the full $\pi$-part of $|G|$ and is of $\pi$-defect zero.  The assertion follows from Lemma~\ref{lem:defect}.

We may assume $n\ge4$.  The Hall classification places $H$ in a monomial subgroup.  By Lemma~\ref{lem:linear-mon}, the character $\chi^{(n-2,2)}$ has an $H$-fixed vector.  Since $q\equiv1\pmod r$, Gaussian binomial coefficients reduce modulo $r$ to ordinary binomial coefficients.  Thus by \eqref{eq:two-row-degree},
\[
\chi^{(n-2,2)}(1)
\equiv \binom n2-n
=\frac{n(n-3)}2
\equiv0\pmod r.
\]
The character occurs in the permutation representation on $2$-spaces, so scalar matrices act trivially.  Its restriction to $\SL_n(q)$ is irreducible, and it therefore gives an irreducible character of $\PSL_n(q)$ whose degree is divisible by $r$ and which has a nonzero $H$-fixed vector.
\end{proof}

\subsection{Two multiplicative orders}

Suppose first that
\[
e(q,r)=r-1,
\qquad
e(q,s)=r\quad(s\in\tau).
\]
For $r\ge5$, choose $1\le b\le r-2$ such that $n-b$ is divisible by neither $r-1$ nor $r$.  At most two values of $b$ are forbidden.  Except when $r=5$ and $n=6$, we may also choose $b\ne n-b$: for $r\ge7$ there are at least five possible values of $b$, while for $r=5$ the only time the unique admissible value can be $n/2$ is $n=6$.  Put $a=n-b$, and let $T^*$ be the corresponding projective torus of type $(a,b)$ in $G^*=\PGL_n(q)$.  Its order is
\[
 |T^*|=\frac{(q^a-1)(q^b-1)}{q-1}.
\]
The order conditions just established show that $T^*$ is a $\pi'$-torus.  By Lemma~\ref{lem:typeA-regular-tori}, choose $s\in T^*$ such that
\[
[C_{G^*}(s):T^*]\mid\gcd(a,b).
\]
Since $b\le r-2$ and $r=\min\pi$, the integer $\gcd(a,b)$ is a $\pi'$-number.  Thus $C_{G^*}(s)$ is a $\pi'$-group, and the associated semisimple character of $G=\PSL_n(q)$ has $\pi$-defect zero.

If $r=5$ and $n=6$, use instead the projective torus of type $(3,2,1)$.  Its order is
\[
(q^3-1)(q^2-1),
\]
which is a $\pi'$-number because the relevant multiplicative orders are $4$ and $5$.  Choose elements of full degrees $3$ and $2$ on the first two blocks and a scalar on the last block, with the three Frobenius orbits disjoint.  Any scalar multiplier preserving their union must preserve each orbit and has order dividing $\gcd(3,2,1)=1$.  Hence the projective centralizer is exactly this torus, and again the associated semisimple character has $\pi$-defect zero.

For $r=3$, the floor conditions give only $n=3$ in the first mixed pattern and $n=2,5$ in the second.  The value $n=2$ cannot occur with a prime $s$ satisfying $e(q,s)=3$.  If $n=5$, the Singer torus is a $\pi'$-torus.  If $n=3$ and $q>2$, a regular split torus is a $\pi'$-torus.  Finally, for $q=2$ we have $\PSL_3(2)\cong\PSL_2(7)$ and $\pi=\{3,7\}$; here no $\pi$-defect-zero character exists, but the degree-$8$ coset action of a subgroup $7{:}3$ has the form
\[
1_H^G=1_G+\chi,
\qquad \chi(1)=7.
\]

Finally suppose that
\[
e(q,r)=r-1,
\qquad
e(q,s)=1\quad(s\in\tau).
\]
The proof of \cite[Theorem~4]{VR02} gives
\[
n=mr+b=m(r-1)+(m+b),
\qquad m+b<r-1,
\]
and hence $r-1\nmid n$.  Moreover the primes in $\tau$ are larger than $n$ in this case.  Therefore the projective Singer torus $T^*\leq G^*=\PGL_n(q)$, of order $(q^n-1)/(q-1)$, is a $\pi'$-torus.  By Lemma~\ref{lem:typeA-regular-tori}, choose $s\in T^*$ with $C_{G^*}(s)=T^*$, and the associated semisimple character of $G=\PSL_n(q)$ has $\pi$-defect zero.

\begin{theorem}\label{thm:linear}
Let $G=\PSL_n(q)$ be simple and let $1<H<G$ with $H\in\Hall_\pi(G)$ and $2,p\notin\pi$.  Then there exists $\chi\in\Irr(1_H^G)$ such that $\chi(1)$ is divisible by a prime in $\pi$.
\end{theorem}

\begin{remark}
For the character $\chi^{(n-2,2)}$, irreducibility on restriction from $\GL_n(q)$ to $\SL_n(q)$ follows from Clifford theory: a nontrivial determinant twist moves a unipotent character to a nontrivial semisimple Lusztig series, so it cannot fix $\chi^{(n-2,2)}$.
\end{remark}

\section{Unitary groups}

Let
\[
G=\PSU_n(q),\qquad q=p^f,
\]
and again assume $2,p\notin\pi$.  For an odd prime $t\ne p$, define
\[
d(t)=\operatorname{ord}_t(-q).
\]
If $a=e(q,t)$, then
\begin{equation}\label{eq:unitary-param}
d(t)=
\begin{cases}
2a,&a\text{ odd},\\
a/2,&a\equiv2\pmod4,\\
a,&a\equiv0\pmod4.
\end{cases}
\end{equation}
Moreover
\begin{equation}\label{eq:unitary-div}
t\mid q^m-(-1)^m
\quad\Longleftrightarrow\quad
d(t)\mid m.
\end{equation}

The nine pairwise cases for $\GU_n(q)$ in the proof of \cite[Theorem~4]{VR02} collapse, in terms of $d(t)$, to the same four patterns as in the linear case:
\[
d(t)=d(s),
\]
or, with $t=r=\min\pi$,
\[
d(r)=r-1,\quad d(s)=r,
\]
or
\[
d(r)=r-1,\quad d(s)=1.
\]
In particular all primes in $\tau$ have one common unitary parameter.  Vdovin and Revin give the unitary cases and state that cases (8) and (9) are treated in the same way as linear case (4) \cite[proof of Theorem~4]{VR02}.

\subsection{Equal orders of $-q$}

\begin{proposition}\label{prop:unitary-equal}
Assume $d(t)=d$ for every $t\in\pi$.  If $d>1$, then $G$ has a $\pi$-defect-zero character.  If $d=1$ and no prime in $\pi$ divides $n$, the same conclusion holds.
\end{proposition}

\begin{proof}
The proof is the unitary analogue of Proposition~\ref{prop:linear-equal}.  We work in the adjoint unitary group $X^*=\PGU_n(q)$, where the regular semisimple label is chosen.  If $d\nmid n$, let $T^*$ be the projective Singer torus.  Then
\[
 |T^*|=\frac{q^n-(-1)^n}{q+1}.
\]
By \eqref{eq:unitary-div}, a prime $t\in\pi$ can divide $q^n-(-1)^n$ only when $d\mid n$.  Thus $T^*$ is a $\pi'$-group.  If $d\mid n$, instead take the projective torus of type $(n-1,1)$.  Its order is
\[
 |T^*|=q^{n-1}-(-1)^{n-1},
\]
and this is again a $\pi'$-number because $d\nmid n-1$.  By Lemma~\ref{lem:typeA-regular-tori}, choose $s\in T^*$ with $C_{X^*}(s)=T^*$, and the regular-semisimple degree formula above gives
\[
 \chi_s(1)=|X^*:T^*|_{p'}.
\]
Since $T^*$ is a $\pi'$-group and $p\notin\pi$, the resulting semisimple character has $\pi$-defect zero; by the central-quotient duality stated above it gives an irreducible character of $G=\PSU_n(q)$ of $\pi$-defect zero.

If $d=1$, then every $t\in\pi$ divides $q+1$, and LTE gives
\[
v_t\left(\frac{q^n-(-1)^n}{q+1}\right)=v_t(n).
\]
Hence the projective Singer torus $T^*\leq\PGU_n(q)$ is a $\pi'$-torus whenever no prime in $\pi$ divides $n$.  By Lemma~\ref{lem:typeA-regular-tori}, $T^*$ contains an element whose projective centralizer is $T^*$; the associated semisimple character has $\pi$-defect zero.
\end{proof}

It remains to consider the equal-$d=1$ case with $r\mid n$.  If $n=3$, then $r=3$.  Choose a regular element $s$ in a Singer torus of $X^*=\PGU_3(q)$.  Then
\[
 |C_{X^*}(s)|=q^2-q+1
\]
and the corresponding semisimple character of $G=\PSU_3(q)$ has degree
\[
 \chi_s(1)=(q-1)(q+1)^2.
\]
Since
\[
 |G|_{p'}=\frac{(q-1)(q+1)^2(q^2-q+1)}{(3,q+1)}
\]
and $3\mid q+1$, LTE gives $v_3(q^2-q+1)=1$.  Also $\gcd(q+1,q^2-q+1)=3$.  Hence $\chi_s(1)$ contains the full $t$-part of $|G|$ for every $t\in\pi$, so $\chi_s$ has $\pi$-defect zero.  We may therefore assume $n\ge5$.

Here we use the rank-three action on isotropic points; see \cite[Theorem~1.1]{KL82}.  Let $\Omega$ be the set of totally isotropic $1$-spaces and let $P$ be a point stabilizer.  Then
\[
1_P^G=1+\alpha+\beta.
\]
The nonprincipal constituent degrees are as follows.  If $n=2m$, then
\begin{align*}
\alpha(1)&=\frac{q^2(q^{2m}-1)(q^{2m-3}+1)}{(q+1)(q^2-1)},\\
\beta(1)&=\frac{q^3(q^{2m-2}-1)(q^{2m-1}+1)}{(q+1)(q^2-1)}.
\end{align*}
If $n=2m+1$, then
\begin{align*}
\alpha(1)&=\frac{q^3(q^{2m}-1)(q^{2m-1}+1)}{(q+1)(q^2-1)},\\
\beta(1)&=\frac{q^2(q^{2m-2}-1)(q^{2m+1}+1)}{(q+1)(q^2-1)}.
\end{align*}
These are the usual multiplicities of the two nontrivial eigenvalues in the Hermitian polar graph.

\begin{lemma}\label{lem:unitary-rank3}
Let $M$ be the normalizer of an orthogonal decomposition of the natural unitary module into nonsingular $1$-spaces.  If $q$ is odd and $n\ge4$, then both $\alpha$ and $\beta$ have nonzero $M$-fixed vectors, except when $(q,n)=(3,5)$.  In that exceptional case $\alpha^M\ne0$, and $\alpha(1)$ is even.  More generally, if $n\ge4$ and an odd prime $r$ divides both $q+1$ and $n$, then both $\alpha$ and $\beta$ have nonzero $M$-fixed vectors; in this case, if $n$ is even then $r\mid\alpha(1)$, while if $n$ is odd then $r\mid\beta(1)$.
\end{lemma}

\begin{proof}
When an odd prime $r$ divides both $q+1$ and $n$, the divisibility follows directly from the displayed degree formulas and LTE.  If $n=2m$, then $r\mid m$ and
\[
v_r(\alpha(1))=v_r(m)+v_r(2m-3)>0,
\qquad
v_r(\beta(1))=0.
\]
If $n=2m+1$, then
\[
v_r(\alpha(1))=0,
\qquad
v_r(\beta(1))=v_r(n)+v_r(m-1)>0.
\]

We prove the fixed-point assertion.  Choose an orthonormal basis $e_1,\ldots,e_n$ defining $M$, and let $O_2$ be the $M$-orbit of isotropic points with support exactly $2$.  Put $f=1_{O_2}\in\mathbb C[\Omega]^M$, and let $A$ be the adjacency operator of the collinearity graph on $\Omega$.

For an isotropic point $x=\langle(x_1,\ldots,x_n)\rangle$ outside $O_2$, put
\[
S=\operatorname{supp}(x),\qquad k=|S|,
\qquad a_i=x_i^{q+1}\in\mathbb F_q^\times\quad(i\in S).
\]
Since $x$ is isotropic, $\sum_{i\in S}a_i=0$.  A direct count gives
\begin{equation}\label{eq:unitary-neighbors}
(Af)(x)=(q+1)\binom{n-k}{2}
 +\#\{\{i,j\}\subseteq S:a_i+a_j=0\}.
\end{equation}
Indeed, a coordinate pair disjoint from $S$ contributes all $q+1$ isotropic points on that hyperbolic $2$-space, a pair contained in $S$ contributes one point precisely when $a_i+a_j=0$, and a mixed pair contributes none.

Suppose first that $q>3$.  In the fixed-vector assertion we have $n\ge4$; in the divisibility assertion $r\mid n$ and hence $n\ge5$.  There is an isotropic point of support $3$; for such a point no two of the three nonzero norms can sum to zero, and hence
\[
(Af)(x_3)=(q+1)\binom{n-3}{2}.
\]
There is also an isotropic point of support $4$ with norm pattern $(a,-a,b,-b)$, where $b\ne\pm a$; then
\[
(Af)(x_4)=(q+1)\binom{n-4}{2}+2.
\]
These values are distinct for $n\ge4$.

If $q=3$, then the assertion involving an odd prime $r\mid q+1$ does not arise.  For the fixed-vector assertion, support $3$ and support $4$ with norm pattern $(a,-a,a,-a)$ give
\[
(Af)(x_3)=4\binom{n-3}{2},
\qquad
(Af)(x_4)=4\binom{n-4}{2}+4.
\]
Their difference is $4(n-5)$, so both nonprincipal projections are nonzero when $n\ne5$.  Suppose that $n=5$.  An isotropic point outside $O_2$ has support $3$, $4$, or $5$; for support $5$ the nonzero norms consist of four copies of $a$ and one copy of $-a$, or the reverse.  Formula~\eqref{eq:unitary-neighbors} therefore gives
\[
(Af)(x)=4\qquad(x\notin O_2).
\]
For $x\in O_2$, the same direct count, with $x$ itself omitted from the adjacency count, gives $(Af)(x)=12$.  Hence
\[
Af=8f+4\mathbf1.
\]
For $n=5$ the two restricted eigenvalues of the Hermitian point graph are $q^2-1=8$ and $-(q^3+1)=-28$.  Comparing their multiplicities with the displayed character degrees shows that the $8$-eigenspace affords $\alpha$.  Hence $\alpha^M\ne0$.  From the displayed degree formula above,
\[
\alpha(1)=1890,
\]
which is even.

If $q=2$, then $r=3$ and isotropic supports have even size.  For support $4$ and support $6$ we obtain from \eqref{eq:unitary-neighbors}
\[
(Af)(x_4)=3\binom{n-4}{2}+6,
\qquad
(Af)(x_6)=3\binom{n-6}{2}+15,
\]
whose difference is $6(n-7)\ne0$ because $3\mid n$ and $n\ge6$.

The pair $(q,n)=(3,5)$ was handled in the preceding paragraph, so assume $(q,n)\ne(3,5)$.  Let $E_\alpha$ and $E_\beta$ denote the two nonprincipal eigenspaces, affording $\alpha$ and $\beta$, respectively.  Then
\[
\mathbb C[\Omega]=\langle\mathbf1\rangle\oplus E_\alpha\oplus E_\beta.
\]
If the projection of $f$ onto either $E_\alpha$ or $E_\beta$ were zero, then $Af$ would be constant on $\Omega\setminus O_2$.  The preceding calculation contradicts this.  Hence $f$ has nonzero projection onto both eigenspaces, and therefore
\[
\alpha^M\ne0,
\qquad
\beta^M\ne0.
\]
\end{proof}

When $d(t)=1$ for every $t\in\pi$, all primes in $\pi$ divide $q+1$.  By the maximal-torus normalizer result for odd primes dividing $q+1$, the unitary monomial normalizer contains a Hall $\pi$-subgroup; since odd-order Hall subgroups are conjugate, we may assume $H\le M$.  Lemma~\ref{lem:unitary-rank3} then gives a constituent of $1_H^G$ whose degree is divisible by $r$.

\subsection{Two orders of $-q$}

Suppose first that
\[
d(r)=r-1,
\qquad d(s)=r\quad(s\in\tau).
\]
For $r\ge5$ use the same choice of $b$ as in the linear case.  If $a=n-b\ne b$, the projective unitary torus of type $(a,b)$ is a $\pi'$-torus, and Lemma~\ref{lem:typeA-regular-tori} gives an element $s$ for which
\[
[C_{X^*}(s):T^*]\mid\gcd(a,b).
\]
Again $b\le r-2$ and $r=\min\pi$, so the whole centralizer is a $\pi'$-group.  In the exceptional choice $r=5$, $n=6$, use the unitary torus of type $(3,2,1)$.  Its projective order is
\[
(q^3+1)(q^2-1),
\]
which is prime to the Hall primes because their unitary parameters are $4$ and $5$; the three different orbit lengths also force the projective centralizer to be this torus.  Thus the corresponding semisimple character is of $\pi$-defect zero.  If $r=3$, the floor conditions give only $n=3$ or $n=5$ (the formal value $n=2$ cannot contain a prime of unitary parameter $3$).  For $n=3$ a regular split unitary torus is $\pi'$, while for $n=5$ the unitary Singer torus is $\pi'$.

Now suppose that
\[
d(r)=r-1,
\qquad d(s)=1\quad(s\in\tau).
\]
The unitary cases (8),(9) in the proof of \cite[Theorem~4]{VR02} satisfy
\[
n<2s,
\qquad
\left\lfloor\frac{n}{r-1}\right\rfloor
=
\left\lfloor\frac nr\right\rfloor.
\]
Vdovin and Revin state that these cases are treated in the same way as linear case (4).  In particular the Hall subgroup is obtained from the analogous block and monomial construction.  Unless $s=n$ for some $s\in\tau$, the unitary Singer torus is a $\pi'$-torus and gives a $\pi$-defect-zero character.

If $s=n$, the Singer torus has an $s$-part.  We need the subgroup description in the proof of \cite[Theorem~4]{VR02}.  In linear case~(4) that proof places a Hall $\pi$-subgroup in a product of monomial subgroups of the $r$-dimensional and $1$-dimensional factors.  After listing the nine unitary cases, Vdovin and Revin state that unitary cases~(8) and~(9) are treated in the same way as linear case~(4).  Thus, in the unitary analogue, a Hall $\pi$-subgroup is contained in the normalizer $M$ of the corresponding orthogonal decomposition into nonsingular $1$-spaces.  Since odd-order Hall subgroups are conjugate, we may assume $H\le M$.  Now $s=n$ is odd and $s\mid q+1$, so Lemma~\ref{lem:unitary-rank3}, applied with the prime $s$, gives an $s$-divisible constituent with an $H$-fixed vector.

\begin{theorem}\label{thm:unitary}
Let $G=\PSU_n(q)$ be simple and let $1<H<G$ with $H\in\Hall_\pi(G)$ and $2,p\notin\pi$.  Then there exists $\chi\in\Irr(1_H^G)$ such that $\chi(1)$ is divisible by a prime in $\pi$.
\end{theorem}

\section{Symplectic groups}

Let
\[
G=\PSp_{2n}(q),\qquad q=p^f,
\]
and suppose that $H\in\Hall_\pi(G)$ with $2,p\notin\pi$.  For $t\in\pi$ put $e_t=e(q,t)$.  The symplectic part of the proof of \cite[Theorem~4]{VR02}, using Gross \cite{Gross95}, shows that for $t<s$ in $\pi$ only the following two possibilities occur:
\[
e_t=e_s\equiv0\pmod2,
\qquad 2n<e_s s,
\]
or
\[
e_t=e_s\equiv1\pmod2,
\qquad n<e_s s.
\]
Thus there is a common value
\[
d=e(q,t)\qquad(t\in\pi).
\]

\begin{proposition}\label{prop:symplectic}
Let $G=\PSp_{2n}(q)$ be simple and let $1<H<G$ with $H\in\Hall_\pi(G)$ and $2,p\notin\pi$.  Then $G$ has a $\pi$-defect-zero character.
\end{proposition}

\begin{proof}
Let $N=G$ and let $X$ be the corresponding inner-diagonal group.  Then $X/N$ is a $2$-group.  It suffices to construct a $\pi$-defect-zero character of $X$; Lemma~\ref{lem:two-extension} then yields a $\pi$-defect-zero constituent on $N$ with a nonzero $H$-fixed vector.

Work in the dual group of type $B_n$.  There are maximal tori whose orders contain the one-block factors $q^n-1$ and $q^n+1$, corresponding to the positive and negative signed partitions of $n$.

Suppose first that $d$ is odd.  If $t\in\pi$ divided $q^n+1$, then $q^n\equiv-1\pmod t$, forcing $e(q,t)$ to be even.  Hence the negative one-block torus is a $\pi'$-torus.

Now suppose that $d$ is even.  For $t\in\pi$,
\[
t\mid q^n+1
\quad\Longleftrightarrow\quad
n\equiv d/2\pmod d.
\]
If this congruence fails, the negative one-block torus is $\pi'$.  If it holds, then $d\nmid n$, so the positive one-block torus is $\pi'$.

These one-block tori contain regular semisimple elements: a generator has the full Frobenius orbit required by the signed cycle.  Hence the corresponding semisimple character of $X$ has $\pi$-defect zero.  Lemma~\ref{lem:two-extension} yields a $\pi$-defect-zero constituent on $N$ with an $H$-fixed vector.
\end{proof}

By Lemma~\ref{lem:defect}, Proposition~\ref{prop:symplectic} gives a constituent of $1_H^G$ whose degree is divisible by a prime in $\pi$.

\section{Orthogonal groups}

Consider now the simple orthogonal groups.  For odd-dimensional groups the dual root system is of type $C$, so the torus argument from the symplectic case applies without change.  Thus only the even-dimensional groups require consideration.

Let $n=2m$ and let $G$ be a simple quotient of an even-dimensional orthogonal group of sign $\varepsilon\in\{+,-\}$.  For $t<s$ in $\pi$ put
\[
a=e(q,t),\qquad b=e(q,s).
\]
The orthogonal part of the proof of \cite[Theorem~4]{VR02} lists exactly six possibilities.  In four of them $a=b$, with the parity of this common value and the sign of the group determining the rank inequality.  The remaining two are
\[
\varepsilon=-,\qquad a\text{ odd},\qquad b=2a,\qquad n=4a,
\]
and
\[
\varepsilon=-,\qquad b\text{ odd},\qquad a=2b,\qquad n=4b.
\]
In the latter two cases the Hall subgroup is cyclic; see the proof of \cite[Theorem~4]{VR02}.

\begin{proposition}\label{prop:orthogonal}
Let $G$ be a finite simple orthogonal group in characteristic $p$, and let $1<H<G$ with $H\in\Hall_\pi(G)$ and $2,p\notin\pi$.  Then $G$ has a $\pi$-defect-zero character.
\end{proposition}

\begin{proof}
Let $N=G$ and let $X$ be the corresponding inner-diagonal orthogonal group.  Then $X/N$ is a $2$-group.  It suffices to construct a $\pi$-defect-zero semisimple character of $X$; Lemma~\ref{lem:two-extension} then yields a $\pi$-defect-zero constituent on $N$ with a nonzero $H$-fixed vector.

We use the signed-partition description of maximal tori in type $D_m$.  A positive part $j$ contributes $q^j-1$ and a negative part contributes $q^j+1$; the sign is determined by the parity of the number of negative parts.

First suppose that all primes in $\pi$ have the same order $d=e(q,t)$.  If $d$ is odd, every negative factor is $\pi'$.  For minus type use the negative one-part torus
\[
q^m+1.
\]
For plus type use the signed partition $(-(m-1),-1)$, whose torus has order
\[
(q^{m-1}+1)(q+1).
\]
The two negative blocks have different lengths unless $m=2$, and $D_2$ is not a simple case.

Now suppose that $d$ is even.  A positive part $j$ can contribute a prime in $\pi$ only when $d\mid j$, while a negative part $j$ can contribute one only when
\[
j\equiv d/2\pmod d.
\]
For plus type, if $d\nmid m$, use the positive one-part torus $q^m-1$.  If $d\mid m$ and $q\ge4$, use the signed partition $(m-1,1)$, giving
\[
(q^{m-1}-1)(q-1).
\]
Both factors are $\pi'$.  If $q=2$ or $3$, the $1$-block need not supply a regular eigenvalue.  When $d\ge4$ replace $(m-1,1)$ by $(-(m-1),-1)$; since $m\equiv0\pmod d$, neither negative part is congruent to $d/2$ modulo $d$.  The only remaining possibility is $q=2$ and $d=2$.  If $m=4$, then $G=P\Omega_8^+(2)\cong O_8^+(2)$.  Here $|G|_3=3^5$, while the character table in the \textit{Atlas}~\cite{Atlas85} contains an irreducible character of degree
\[
972=4\cdot3^5.
\]
In this case necessarily $\pi=\{3\}$, since the unique odd prime of multiplicative order $2$ modulo $2$ is $3$.  Thus this is a $\pi$-defect-zero character, and Lemma~\ref{lem:defect} applies.  Hence assume $m\ne4$.  Use $(-2,-(m-2))$; both exponents are even, so neither factor is divisible by the unique odd prime of order $2$ modulo $2$, namely $3$.  The case $m=2$ is nonsimple, and hence in the remaining cases the two negative blocks have distinct lengths.

For minus type, if $m\not\equiv d/2\pmod d$, use the negative one-part torus $q^m+1$.  Suppose
\[
m\equiv d/2\pmod d.
\]
If $q\ge4$, use the signed partition $(-(m-1),1)$, giving
\[
(q^{m-1}+1)(q-1).
\]
If $q=3$, then $d=2$ cannot occur for an odd prime, and we use $(-(m-2),2)$ instead.  Its torus has order
\[
(q^{m-2}+1)(q^2-1),
\]
and both factors are $\pi'$.  If $q=2$ and $d>2$, use $(-(m-3),3)$, giving
\[
(q^{m-3}+1)(q^3-1);
\]
again both factors are $\pi'$.  Finally, if $q=2$ and $d=2$, then $m$ is odd and we use $(-2,m-2)$, whose torus has order
\[
(q^2+1)(q^{m-2}-1).
\]
This is prime to $3$.  The unavailable ranks $m\le3$ are $D_2$ or $D_3$; the latter is covered by the linear and unitary low-rank isomorphisms.

There are two mixed minus-type cases left.  If $a$ is odd, $b=2a$, and $n=4a$, then $m=2a$, and the negative one-part torus has order $q^{2a}+1$.  Since both relevant multiplicative orders divide $2a$, this number is congruent to $2$ modulo both Hall primes.  The case $b$ odd, $a=2b$, and $n=4b$ is identical.

It remains to check regularity for the elements used above, apart from the case $P\Omega_8^+(2)$ just settled directly.  A generator of a negative block of length $j$ has Frobenius orbit of length $2j$, with inversion first occurring after $j$ steps.  For the positive blocks occurring in the two-block constructions, a generator has Frobenius orbit of length $j$ disjoint from its inverse orbit; the replacements for $q=2,3$ avoid the degenerate positive $1$-blocks.  The two block lengths are distinct in every remaining two-block construction, so their eigenvalue orbits are disjoint.  Thus each torus still under consideration contains a regular semisimple element.  Hence $X$ has a semisimple character of $\pi$-defect zero.  Lemma~\ref{lem:two-extension} gives a $\pi$-defect-zero constituent on $N$ with an $H$-fixed vector.
\end{proof}

\section{Hall subgroups in defining characteristic}

Suppose that $p\in\pi$ and $2\notin\pi$.  The following argument does not depend on the classical type.

\begin{proposition}\label{prop:defining}
Let $G$ be a finite simple group of Lie type in characteristic $p$, and let $H\in\Hall_\pi(G)$ with $p\in\pi$ and $2\notin\pi$.  Then the Steinberg character of $G$ is a constituent of $1_H^G$.
\end{proposition}

\begin{proof}
By Revin's classification of Hall subgroups in defining characteristic, a Hall $\pi$-subgroup is either contained in a Borel subgroup or is a parabolic subgroup; see \cite[Theorem~8.3]{VRsurvey} and \cite{Revin99}.  Since $p\in\pi$ and $2\notin\pi$, we have $p\ne2$, so the defining field has odd order.  If the parabolic subgroup $H$ is not a Borel subgroup, then the Lie rank is at least $2$, and $H$ contains a torus factor of even order.  This is impossible because $H$ has odd order.  Thus in either case
\[
H\le B
\]
for some Borel subgroup $B$ of $G$.

Let $\St$ be the Steinberg character of $G$.  The Steinberg representation has a one-dimensional space of $B$-fixed vectors; equivalently,
\[
[\St,1_B^G]=1;
\]
see, for example, \cite[Chapter~6]{Carter85}.  Thus
\[
0<\dim \St^B\le \dim \St^H,
\]
and Frobenius reciprocity gives
\[
[\St_H,1_H]>0.
\]
Therefore $\St\in\Irr(1_H^G)$.  Finally,
\[
\St(1)=|G|_p,
\]
so $p\mid\St(1)$.
\end{proof}

\begin{theorem}\label{thm:classical}
Let $G$ be a finite simple classical group in defining characteristic $p$, and let $1<H<G$ with $H\in\Hall_\pi(G)$ and $2\notin\pi$.  Then there exists
\[
\chi\in\Irr(1_H^G)
\]
such that $\chi(1)$ is divisible by a prime in $\pi$.
\end{theorem}

\begin{proof}
If $p\notin\pi$, the linear and unitary groups are Theorems~\ref{thm:linear} and \ref{thm:unitary}, while the symplectic and orthogonal groups follow from Propositions~\ref{prop:symplectic} and \ref{prop:orthogonal}.  If $p\in\pi$, apply Proposition~\ref{prop:defining}.
\end{proof}

\section{Hall subgroups containing $2$}

Assume that $2\in\pi$.  The elementary observation below will be used repeatedly.

\begin{lemma}\label{lem:overgroup}
Let $H\le M\le G$.  Then every irreducible constituent of $1_M^G$ is a constituent of $1_H^G$.
\end{lemma}

\begin{proof}
The natural surjection $G/H\to G/M$ gives, by pullback of functions, an injective $G$-module map
\[
\mathbb C[G/M]\hookrightarrow \mathbb C[G/H].
\]
Thus every irreducible constituent of the permutation module on $G/M$ also occurs in the permutation module on $G/H$.

Equivalently, by Frobenius reciprocity,
\[
[(1_H)^M,1_M]_M=[1_H,1_H]_H=1.
\]
Hence the principal character occurs exactly once in $(1_H)^M$, and we may write
\[
(1_H)^M=1_M+\Theta
\]
for a character $\Theta$ of $M$.  (If $H=M$, then $\Theta=0$; if $H<M$, then
$\Theta(1)=[M:H]-1>0$.)  By transitivity of induction,
\[
(1_H)^G=((1_H)^M)^G=(1_M)^G+\Theta^G.
\]
Therefore every irreducible constituent of $1_M^G$ is a constituent of $1_H^G$.
\end{proof}

\subsection{The defining-characteristic case}

Revin and Vdovin summarize the defining-characteristic classification needed here in the proof of \cite[Lemma~8.1]{RV10}.  Apart from the trivial case in which $\pi(G)\subseteq\pi$, a proper Hall subgroup falls into one of three types: it is contained in a Borel subgroup; it is the special orthogonal parabolic in characteristic $2$; or $G=\PSL_n(q)$ and it is one of the listed flag parabolics.

\begin{proposition}\label{prop:defining-even}
Let $G$ be a finite simple classical group in characteristic $p$, and let
\[
1<H<G,\qquad H\in\Hall_\pi(G),\qquad p,2\in\pi.
\]
Then $1_H^G$ has a nonprincipal irreducible constituent whose degree is divisible by $p$.
\end{proposition}

\begin{proof}
Suppose first that $H$ is contained in a Borel subgroup $B$.  Then the Steinberg character satisfies
\[
[\St,1_B^G]=1,
\]
and hence $\St$ is a constituent of $1_H^G$ by Lemma~\ref{lem:overgroup}.  Since $\St(1)=|G|_p$, this gives the result.

Next suppose that the special orthogonal alternative occurs.  Then $p=2$ and
\[
G=P\Omega_{2n}^{\eta}(q),
\]
while $H$ is the maximal parabolic stabilizing a singular $1$-space.  The action on singular $1$-spaces has rank three,
\[
1_H^G=1+\alpha+\beta.
\]
Since the action is transitive of rank three, $\alpha$ and $\beta$ are distinct nonprincipal irreducible characters.
The multiplicities are as follows.  For plus type,
\[
\alpha(1)=\frac{q(q^{n-2}+1)(q^n-1)}{q^2-1},
\qquad
\beta(1)=\frac{q^2(q^{2n-2}-1)}{q^2-1},
\]
while for minus type,
\[
\alpha(1)=\frac{q^2(q^{2n-2}-1)}{q^2-1},
\qquad
\beta(1)=\frac{q(q^{n-2}-1)(q^n+1)}{q^2-1}.
\]
These are obtained from the usual parameters of the orthogonal polar graph.  Since $q$ is even, both nonprincipal degrees are divisible by $q$, and hence by $p=2$.

Finally suppose that $G=\PSL_n(q)$ and $H$ is one of the flag parabolics in the defining-characteristic list.  Let $d$ be the dimension of a proper subspace occurring in the flag and let $P_d$ be its stabilizer.  Then $H\le P_d$.  The Grassmann permutation character $1_{P_d}^G$ contains the unipotent character labelled by $(n-1,1)$, whose degree is
\[
q\cdot\frac{q^{n-1}-1}{q-1}.
\]
It is therefore divisible by $p$, and Lemma~\ref{lem:overgroup} completes the proof.
\end{proof}

\begin{remark}\label{rem:defining-list}
In the notation of \cite[Lemma~8.1]{RV10}, the proof above treats Cases~2--4 simultaneously.  Case~1 has $\pi(G)\subseteq\pi$ and hence has no proper Hall $\pi$-subgroup.
\end{remark}

\subsection{The case $2,3\in\pi$ and $p\notin\pi$}

Assume now that $p\notin\pi$ and $2,3\in\pi$.  Revin--Vdovin give a decomposition-theoretic description of the Hall subgroups in the classical groups; see \cite[Lemmas~4.3, 4.4 and 6.7]{RV10}.  We begin with the symplectic groups.

\begin{proposition}\label{prop:even-sp}
Let $G=\PSp_{2n}(q)$ be simple, with $q$ odd, and let $H\in\Hall_\pi(G)$, where $2,3\in\pi$ and $p\notin\pi$.  Then $1_H^G$ has an even-degree irreducible constituent.
\end{proposition}

\begin{proof}
By \cite[Lemma~4.4]{RV10}, $H$ is contained in the image $M$ of
\[
\Sp_2(q)\wr S_n.
\]
Let $G$ act on the projective points of its natural symplectic module, and let $P$ be a point stabilizer.  This is a rank-three action,
\[
1_P^G=1+\alpha+\beta,
\]
with
\[
\alpha(1)=\frac{q(q^{n-1}+1)(q^n-1)}{2(q-1)},\qquad
\beta(1)=\frac{q(q^{n-1}-1)(q^n+1)}{2(q-1)}.
\]
For odd $q$, $\alpha(1)$ is even when $n$ is even and $\beta(1)$ is even when $n$ is odd.

The subgroup $M$ preserves the decomposition into $n$ nondegenerate $2$-spaces.  The support of a point is the set of blocks on which a representative has nonzero component.  Let $O_1$ be the set of projective points supported in one block, put $f=1_{O_1}$, and let $A$ be the adjacency operator of the rank-three graph.  Thus $(Af)(x)$ is the number of neighbors of $x$ lying in $O_1$.  If $x$ has support $k\ge2$, a direct count gives
\[
(Af)(x)=n(q+1)-kq.
\]  Hence $Af$ takes at least two values outside $O_1$ when $n\ge3$.  Since $A$ has distinct eigenvalues on the $\alpha$- and $\beta$-spaces, both $\alpha^M$ and $\beta^M$ are nonzero.  For $n=2$ the same calculation gives $\alpha^M\ne0$, and $\alpha(1)$ is even.  Lemma~\ref{lem:overgroup} now gives the result.
\end{proof}

Consider the orthogonal decomposition families.  Write $\varepsilon=1$ if $q\equiv1\pmod4$ and $\varepsilon=-1$ if $q\equiv-1\pmod4$.  The decomposition cases in \cite[Lemma~6.7(a)--(e)]{RV10} place the Hall subgroup in a stabilizer built from repeated nondegenerate $2$-spaces of type $\varepsilon$, together with a residual orthogonal summand.  Cases (f)--(h) require separate arguments below.

\begin{proposition}\label{prop:even-orth-large}
Let $G$ be a simple orthogonal group over a field of odd order $q$, and let $H\in\Hall_\pi(G)$ with $2,3\in\pi$ and $p\notin\pi$.  If $H$ is one of the decomposition-type Hall subgroups in \cite[Lemma~6.7(a)--(e)]{RV10}, then $1_H^G$ has an even-degree irreducible constituent.
\end{proposition}

\begin{proof}
Let $\Omega$ be the singular $1$-spaces, let $P$ be a point stabilizer, and let $A$ be the adjacency operator of the corresponding rank-three graph.  Then
\[
1_P^G=1+\alpha+\beta.
\]
The rank-three decomposition and the two nonprincipal multiplicities are those of the corresponding orthogonal action in \cite[Theorem~1.1]{KL82}.  Substitution in the standard multiplicity formulas for a rank-three graph shows that, for odd $q$, exactly one of $\alpha(1)$ and $\beta(1)$ is even.  Let $M$ be the decomposition stabilizer containing $H$, and let $m$ be the number of repeated nondegenerate $2$-spaces in the decomposition.  The support of a point means the set of repeated $2$-space components on which it is nonzero.  It is therefore enough to show that both $\alpha$ and $\beta$ have nonzero $M$-fixed vectors.

Suppose first that $q\equiv1\pmod{12}$.  The repeated $2$-spaces are hyperbolic.  Let $O$ be the set of singular points contained in one repeated block and put $f=1_O$.  If $x$ has two nonzero singular components, then
\[
(Af)(x)=2m-2,
\]
whereas a singular point $y$ with two anisotropic components of opposite norm satisfies
\[
(Af)(y)=2m-4.
\]
Both points lie outside $O$.  Thus $Af\notin\langle 1,f\rangle$, and the two nonprincipal eigenspace projections of $f$ are both nonzero.  Hence $\alpha^M,\beta^M\ne0$.

Now suppose that $q\equiv-1\pmod{12}$.  The repeated $2$-spaces are anisotropic.  Let $O$ be the set of singular points with support exactly two in the repeated blocks.  For a singular point $x$ with support $s$, put
\[
z(x)=\#\{\{i,j\}:Q(x_i)+Q(x_j)=0\}.
\]
A block-pair count gives
\[
(Af)(x)=(q+1)\left[\binom{m-s}{2}(q+1)+s(m-s)+\binom{s}{2}\right]+qz(x).
\]
If there are at least four repeated blocks, choose a singular point of support three, for which $z(x)=0$, and another of support four with component norms $a,-a,b,-b$, where $b\ne\pm a$; then $z(y)=2$.  Their adjacency counts differ by
\[
q\bigl(mq+m-4q-6\bigr),
\]
which is nonzero in the relevant range.

It remains to handle the case of three repeated anisotropic blocks.  With three repeated anisotropic blocks, write
\[
V=A_1\perp A_2\perp A_3\perp W.
\]
Choose a singular point $x$ with nonzero components in all three $A_i$, zero $W$-component, and component norms $a,b,c\ne0$ satisfying $a+b+c=0$.  Then $s=3$, $z(x)=0$, and
\[
(Af)(x)=3(q+1).
\]
Choose $w\in W$ with nonzero quadratic value.  Such a vector exists in the odd-dimensional row, where $W$ is nonsingular of dimension $1$, and in the even-dimensional mixed row, where the residual nondegenerate $2$-space, of either type, represents a nonzero value.  Since an anisotropic repeated $2$-space represents every nonzero field element, we may choose a component in one repeated block with the opposite quadratic value.  The resulting point $y$ is singular and has one nonzero repeated-block component and a nonzero $W$-component.  Then $s=1$, $z(y)=0$, and
\[
(Af)(y)=(q+1)(q+3).
\]
Thus $Af\notin\langle1,f\rangle$.  The low-dimensional cases follow from the standard isomorphisms.  By \cite[Lemma~5.1 and its proof]{RV10},
\[
\Omega_3(q)\cong\PSL_2(q),\qquad
\Omega_4^-(q)\cong\PSL_2(q^2),\qquad
\Omega_5(q)\cong\PSp_4(q),
\]
and
\[
P\Omega_6^\eta(q)\cong\PSL_4^\eta(q).
\]
These are therefore covered by the rank-one, symplectic, and linear or unitary cases already proved.  The plus-type group in dimension $4$ is of type $D_2=A_1+A_1$ and is not a nonabelian simple classical group, so it does not occur in the present theorem.  Thus no additional low-dimensional orthogonal case remains.

Finally, in case \cite[Lemma~6.7(d)]{RV10} the dimension is $11$, while in case \cite[Lemma~6.7(e)]{RV10} it is $12$ of minus type.  In both cases the stabilizer has four repeated $2$-spaces of type $\varepsilon$; the remaining summand is a sum of nonsingular $1$-spaces.  The preceding support calculation applies with all residual components set equal to zero.  Hence in every decomposition case (a)--(e), both nonprincipal rank-three constituents have fixed vectors, and the even one occurs in $1_H^G$ by Lemma~\ref{lem:overgroup}.
\end{proof}

\subsection{The groups $\PSL_2(q)$}

The group $\PSL_2(q)$ is not covered by the decomposition lemmas for
$\SL_n(q)$ with $n>2$, so we treat it separately.  Put
\[
\varepsilon=\varepsilon(q)\in\{1,-1\},\qquad q\equiv\varepsilon\pmod4.
\]
Thus the maximal torus whose normalizer contains a Sylow $2$-subgroup has
order $(q-\varepsilon)/2$ in $G=\PSL_2(q)$ and its normalizer has order
$q-\varepsilon$.

\begin{proposition}\label{prop:even-A1}
Let $G=\PSL_2(q)$ be simple, let $p\notin\pi$, and suppose that
$2\in\pi$.  If $H\in\Hall_\pi(G)$ is proper, then $1_H^G$ has an
even-degree irreducible constituent.
\end{proposition}

\begin{proof}
Since $2\in\pi$ and $p\notin\pi$, the field order $q$ is odd.  We use the
standard character table of $\PSL_2(q)$; see, for example,
\cite[Chapter~6]{Carter85}.  There is an irreducible character $\chi$ of
degree
\[
\chi(1)=q-\varepsilon.
\]

Suppose first that $3\notin\pi$.  By \cite[Theorem~5.2]{VRsurvey}, every odd prime in $\pi\cap\pi(G)$ divides
$q-\varepsilon$.  The same is true for the full $2$-part of $|G|$, since
$q\equiv\varepsilon\pmod4$.  Hence
\[
\chi(1)_\pi=(q-\varepsilon)_\pi=|G|_\pi=|H|.
\]
Thus $\chi$ has $\pi$-defect zero, and Lemma~\ref{lem:defect} proves the
assertion.  This also covers the dihedral row in the
rank-one part of the $2,3\in\pi$ classification whenever all Hall primes
divide $q-\varepsilon$.

Consider the exceptional Hall subgroups $A_4,S_4,A_5$ in
the rank-one classification.  For these cases choose the standard
character of degree $q-\varepsilon$ from the series opposite to the
torus of sign $\varepsilon$.  It vanishes on the involutions of sign
$\varepsilon$; its values on the opposite torus are parametrized by a
linear character $\vartheta$.  In the $S_4$ case it also vanishes on elements of order $4$.
If an odd prime $\ell$ divides $q-\varepsilon$, then the $\ell$-elements
have sign $\varepsilon$ and contribute zero.  If
$\ell\mid q+\varepsilon$, then on a cyclic subgroup $C$ of order $\ell$
one has
\[
\sum_{1\ne x\in C}\chi(x)
 =-\varepsilon\sum_{1\ne x\in C}
   \bigl(\vartheta(x)+\vartheta(x)^{-1}\bigr).
\]
This sum is $2\varepsilon$ when $\vartheta_C\ne1$, and is
$-2\varepsilon(\ell-1)$ when $\vartheta_C=1$.

For $H\cong A_4$ there are four subgroups of order $3$.  If
$3\mid q-\varepsilon$, then
\[
[\chi_H,1_H]=\frac{q-\varepsilon}{12}>0.
\]
If $3\mid q+\varepsilon$, choose $\vartheta$ nontrivial on a subgroup of
order $3$.  Then
\[
[\chi_H,1_H]=\frac{q+7\varepsilon}{12}>0.
\]
The Hall condition $(q^2-1)_{\{2,3\}}=24$ excludes the only possible
nonpositive value in the second expression.

For $H\cong S_4$ the preceding calculation gives the same two possibilities, since the elements of
orders $2$ and $4$ contribute zero.  Hence
\[
[\chi_H,1_H]=\frac{q-\varepsilon}{24}
\quad\text{or}\quad
[\chi_H,1_H]=\frac{q+7\varepsilon}{24}.
\]
These numbers are positive under the Hall condition
$(q^2-1)_{\{2,3\}}=48$, except when $q=7$.  In that case
$\PSL_2(7)$ has an $S_4$-subgroup of index $7$, and its coset character is
$1_G+\psi$ with $\psi(1)=6$.

Finally let $H\cong A_5$.  The group $A_5$ has ten cyclic subgroups of
order $3$ and six cyclic subgroups of order $5$.  If $\varepsilon=1$, we
choose the opposite-torus parameter $\vartheta$ nontrivial on every
subgroup of order $3$ or $5$ which occurs in the opposite torus.  Each
opposite order-$3$ family then contributes $20$ to the sum of $\chi$ on
$H$, and each opposite order-$5$ family contributes $12$.  Thus
$[\chi_H,1_H]>0$.

Suppose $\varepsilon=-1$.  Whenever possible choose $\vartheta$ nontrivial
but trivial on the $3$- and $5$-parts of the opposite torus.  The
contribution of an opposite order-$3$ family is then $40$, and that of an
opposite order-$5$ family is $48$, so again $[\chi_H,1_H]>0$.  If no such
parameter exists, every prime divisor of the opposite torus $(q-1)/2$
belongs to $\{3,5\}$.  The Hall condition gives
\[
(q^2-1)_{\{2,3,5\}}=120.
\]
Since $q\equiv-1\pmod4$, the integer $(q-1)/2$ is odd.  Also
$\gcd(q-1,q+1)=2$, so the $3$-part and the $5$-part of $(q-1)/2$ are each
at most $3$ and $5$, respectively.  Hence
\[
\frac{q-1}{2}\in\{1,3,5,15\}.
\]
The Hall condition leaves only $q=11$.  Here $H\cong A_5$ has
index $11$ and the standard degree-$11$ permutation character of
$\PSL_2(11)$ is
\[
1_H^G=1_G+\psi,
\qquad \psi(1)=10.
\]
\end{proof}

\subsection{Linear and unitary decomposition subgroups}

We continue to assume that $p\notin\pi$ and $2,3\in\pi$.  The structural point in type $A$ is that, apart from the isolated degree-four exception in \cite[Lemma~4.3(d)]{RV10}, every Hall subgroup is contained in a stabilizer of a decomposition into nonsingular blocks of dimension one or two; the mixed case of degree $11$ in \cite[Lemma~4.3(e)]{RV10} is of the same form.

\begin{lemma}\label{lem:linear-decomp-fixed}
Let $q$ be odd and let $W$ be either the full monomial subgroup of $\GL_n(q)$, or, when $n=2m$, the stabilizer $\GL_2(q)\wr S_m$ of a decomposition into $m$ two-dimensional blocks.  Write
\[
\gamma=\chi^{(n-1,1)},\qquad
\delta=\chi^{(n-2,2)},\qquad
\epsilon=\chi^{(n-2,1,1)}.
\]
Then $\gamma^W\ne0$.  Moreover $\delta^W\ne0$ when $n\ge4$, and $\epsilon^W\ne0$ when $n\ge4$.
\end{lemma}

\begin{proof}
The point permutation character is $1+\gamma$.  In either decomposition action the number of $W$-orbits on projective points is the number of possible support sizes, and is at least two, so $\gamma^W\ne0$.

Let $P_2$ denote the stabilizer of a $2$-space.  For $2$-spaces, write $a_1$ and $a_2$ for the numbers of $W$-orbits on points and on $2$-spaces.  Since
\[
1_{P_2}^{\GL_n(q)}=1+\gamma+\delta,
\]
we have $\dim\delta^W=a_2-a_1$.  In the one-dimensional monomial case Lemma~\ref{lem:linear-mon} gives $a_2>a_1$.  In the two-dimensional block case $a_1=m$.  There are $m$ distinct $2$-space orbits obtained by prescribing the number of blocks in the support, and there is one further orbit with projection dimensions $(2,1)$ on two blocks, whereas the support-two orbit above has projection dimensions $(1,1)$.  Hence $a_2\ge m+1>a_1$.

Finally let $a_F$ be the number of $W$-orbits on flags $L<U$ with $\dim L=1$ and $\dim U=2$.  The flag permutation character is
\[
1+2\gamma+\delta+\epsilon,
\]
so
\[
\dim\epsilon^W=a_F-a_2-a_1+1.
\]
Every $2$-space orbit contributes at least one flag orbit.  In the one-dimensional monomial case choose the basis $e_1,\ldots,e_n$ with respect to which $W$ is monomial.  Then the coordinate plane, the spaces
\[
\langle e_1,e_2+\cdots+e_s\rangle\quad(3\le s\le n),
\]
and the two spaces used in Lemma~\ref{lem:linear-mon} each contain lines of two different support sizes.  Thus $a_F-a_2\ge n+1>a_1-1$.

In the two-dimensional block case choose a nonzero vector $e_i$ in the $i$th block for $1\le i\le m$.  For each support size $2\le s\le m$ the space
\[
\langle e_1,e_2+\cdots+e_s\rangle
\]
contains lines of two different block-support sizes, giving $m-1$ extra flag orbits.  The additional $2$-space orbit with projection dimensions $(2,1)$ also contains a support-one line and a support-two line, giving one more.  Hence $a_F-a_2\ge m=a_1$, and again $\epsilon^W\ne0$.
\end{proof}

\begin{proposition}\label{prop:even-A-large}
Let $G=\PSL_n(q)$ or $\PSU_n(q)$ be simple, let $p\notin\pi$, and assume that $2,3\in\pi$.  Suppose that $H\in\Hall_\pi(G)$ is contained in one of the decomposition subgroups in \cite[Lemma~4.3(b),(c),(e)]{RV10}.  Then $1_H^G$ has an even-degree irreducible constituent.
\end{proposition}

\begin{proof}
We treat the linear and unitary groups separately.

Assume first that $G=\PSL_n(q)$.  In the one-dimensional decomposition row, let $W$ be the full monomial subgroup.  If $n$ is odd, the deleted point constituent
\[
\gamma=\chi^{(n-1,1)},\qquad
\gamma(1)=\frac{q^n-1}{q-1}-1,
\]
has even degree and has a $W$-fixed vector by Lemma~\ref{lem:linear-decomp-fixed}.  If $4\mid n$, use
\[
\delta=\chi^{(n-2,2)}.
\]
Its degree is even because, for odd $q$,
\[
\delta(1)\equiv \binom n2-n=\frac{n(n-3)}2\pmod2,
\]
and Lemma~\ref{lem:linear-decomp-fixed} gives $\delta^W\ne0$.  If $n\equiv2\pmod4$, use
\[
\epsilon=\chi^{(n-2,1,1)}.
\]
The hook formula gives
\[
\epsilon(1)=q^3\cdot\frac{(q^{n-2}-1)(q^{n-1}-1)}{(q^2-1)(q-1)}.
\]
Since $n-2$ is divisible by $4$, the $2$-adic LTE formula shows that this degree is even, and Lemma~\ref{lem:linear-decomp-fixed} gives $\epsilon^W\ne0$.

For the two-dimensional block row of \cite[Lemma~4.3(c)]{RV10}, write $n=2m+k$ with $k\in\{0,1\}$.  If $k=0$, take $W=\GL_2(q)\wr S_m$; when $4\mid n$ use $\delta$, and when $n\equiv2\pmod4$ use $\epsilon$.  Both have $W$-fixed vectors by Lemma~\ref{lem:linear-decomp-fixed}.  If $k=1$, then $n$ is odd.  The decomposition stabilizer has at least two orbits on projective points, one represented by the residual line and one by a point in a $2$-block.  Hence $\gamma$ has a fixed vector, and $\gamma(1)$ is even.  The same argument applies to the mixed row $n=11$:  its decomposition stabilizer has more than one orbit on projective points, so $\gamma$ has a fixed vector, and $\gamma(1)$ is even.  The low-dimensional case $n=3$ in row (c) may instead be treated with a Singer torus: all Hall primes divide $q^2-1$, while the projective Singer factor is prime to them under the congruence in \cite[Lemma~4.3(c)]{RV10}.

Now let $G=\PSU_n(q)$.  In the one-dimensional nonsingular decomposition row, the case $n=3$ is separate.  Here $q\equiv-1\pmod{12}$ and every Hall prime divides $q+1$ (apart from the primes already occurring in $S_3$).  Choose a regular element $s$ in a Singer torus of $X^*=\PGU_3(q)$.  Its centralizer has order $q^2-q+1$, and the corresponding semisimple character of $G$ has degree
\[
 \chi_s(1)=(q-1)(q+1)^2.
\]
Since
\[
 |G|_{p'}=\frac{(q-1)(q+1)^2(q^2-q+1)}{3}
\]
and $v_3(q^2-q+1)=1$, while $\gcd(q+1,q^2-q+1)=3$, this degree contains the full part of $|G|$ at every Hall prime.  Thus $\chi_s$ has $\pi$-defect zero.  Hence assume $n\ge4$.  For the one-dimensional nonsingular decomposition Lemma~\ref{lem:unitary-rank3} applies.  Except for $(q,n)=(3,5)$ both nonprincipal constituents have fixed vectors and exactly one has even degree; in the exceptional case the lemma gives directly an even-degree constituent with a fixed vector.

Suppose next that the decomposition has $m\ge2$ nondegenerate $2$-spaces, with at most one residual nonsingular line.  Let $O$ be the orbit of isotropic points contained in one $2$-space, put $f=1_O$, and let $A$ be the adjacency operator of the isotropic-point rank-three graph.  Thus $(Af)(x)$ is the number of neighbors of $x$ lying in $O$.  This number is obtained block by block.  A zero block contributes all $q+1$ isotropic lines in that block; a nonzero anisotropic component contributes none; and a nonzero isotropic component contributes its own line.  If $m\ge3$, choose an isotropic point $x$ supported on two blocks with two anisotropic components of opposite norms, and an isotropic point $y$ supported on three blocks with one isotropic component and two anisotropic components of opposite norms.  Then
\[
(Af)(x)=(m-2)(q+1),
\qquad
(Af)(y)=(m-3)(q+1)+1,
\]
which are distinct.  If $m=2$, take instead two support-two isotropic points, one with anisotropic components of opposite norms and one with two isotropic components; their adjacency counts are $0$ and $2$.  Thus in every case
\[
Af\notin\langle1,f\rangle,
\]
so both nonprincipal rank-three constituents have fixed vectors.  Again exactly one has even degree.  The mixed row $n=11$ contains four such $2$-spaces, so the preceding adjacency calculation remains valid with all residual coordinates equal to zero.  For $n=3$, the unitary Singer torus has projective order prime to the Hall primes in row (c).

For the linear groups, the characters used above are unipotent characters occurring in projective subspace or flag permutation modules.  Hence scalars act trivially, and their restrictions from $\GL_n(q)$ to $\SL_n(q)$ are irreducible by the same determinant-twist argument as in the remark following Theorem~\ref{thm:linear}.  They therefore define irreducible characters of $\PSL_n(q)$.  In all cases the relevant decomposition stabilizer contains $H$, and Lemma~\ref{lem:overgroup} completes the proof.
\end{proof}

\subsection{The exceptional groups $\PSL_4^\eta(q)$}

The only case in \cite[Lemma~4.3]{RV10} not contained in the decomposition rows is the degree-four exception.

\begin{proposition}\label{prop:even-A6}
Let $G=\PSL_4^\eta(q)$ and suppose that $H\in\Hall_\pi(G)$ is the exceptional subgroup in \cite[Lemma~4.3(d)]{RV10}.  Thus
\[
\pi\cap\pi(G)=\{2,3,5\},\qquad q\equiv5\eta\pmod8,
\]
\[
(q+\eta)_3=3,\qquad (q^2+1)_5=5,
\]
and projectively $H\cong2^4.A_6$.  Then $1_H^G$ has an even-degree irreducible constituent.
\end{proposition}

\begin{proof}

Suppose first that $\eta=1$, so $G=\PSL_4(q)$.  Let $V$ be the natural $4$-dimensional module, and for $i=1,2$ let $P_i$ be the stabilizer of an $i$-dimensional subspace of $V$.  Let $\chi$ be the unipotent character labelled by $(2,2)$.  From the point and $2$-space permutation characters,
\[
1_{P_2}^G-1_{P_1}^G=\chi,
\]
so
\[
\chi(1)=q^2(q^2+1)
\]
and, for every semisimple element $g$ on the natural $4$-space,
\[
\chi(g)=|\operatorname{Fix}_{\mathrm{Gr}_2(V)}(g)|-|\operatorname{Fix}_{\mathbb P(V)}(g)|.
\]
Since $p\notin\{2,3,5\}$, every element of $H$ is semisimple.  We claim that
\[
\chi(g)\ge1-q
\]
for every non-scalar $g$.  This follows from the primary decomposition of the natural module.  An irreducible quartic block gives value $0$, while a cubic block together with a linear block gives value $-1$.  Two irreducible quadratic blocks give a nonnegative value, as does a repeated irreducible quadratic block.  For a quadratic block together with two linear blocks the value is $0$ when the two eigenvalues are distinct.  If they are equal, there are exactly two invariant $2$-spaces and $q+1$ invariant projective points, so the value is $1-q$.

Finally consider elements whose eigenvalues all lie in the ground field.  For eigenspace dimensions $3+1$, $2+2$, and $2+1+1$, a direct count gives respectively a nonnegative value, $q^2+1$, and at least $q+1$.  If the four eigenvalues are distinct, the value is $6-4=2$.  The scalar case is excluded.  This proves the claim.

The projective action of $H$ is faithful, so only the identity is scalar, and
\[
|H|=|2^4.A_6|=5760.
\]
Hence
\[
[\chi_H,1_H]
\ge
\frac{q^2(q^2+1)-5759(q-1)}{5760}.
\]
The Hall arithmetic gives
\[
q\equiv53,77\pmod{120}.
\]
Thus $q\ge53$, and
\[
q^2(q^2+1)>5759(q-1).
\]
Therefore $[\chi_H,1_H]>0$.  Since $\chi(1)$ is even, the result follows in the linear case.

Now suppose that $\eta=-1$, so $G=\PSU_4(q)$.  We use the rank-three action on the totally isotropic $1$-spaces of the natural Hermitian space.  Its permutation character is
\[
1_P^G=1+\alpha+\beta.
\]
For the collinearity graph the parameters needed below are
\[
v=(q^3+1)(q^2+1),\qquad k=q^2(q+1),
\]
and the two restricted eigenvalues are
\[
r=q^2-1,
\qquad
s=-(q+1).
\]
The $r$-eigenspace affords an irreducible constituent $\alpha$ of degree
\[
\alpha(1)=q^2(q^2+1),
\]
which is even because $q$ is odd.  These are the rank-three parameters and multiplicities for the Hermitian polar space; see \cite[Theorem~1.1]{KL82}.

Apply Lemma~\ref{lem:rank3-orbit-div} to the subgroup $H$.  Since
\[
k-s=(q+1)(q^2+1)
\]
and
\[
v=(q+1)(q^2-q+1)(q^2+1),
\]
we have
\[
D_r=\frac{v}{\gcd(v,k-s)}=q^2-q+1.
\]
Here $|H|=5760$.  The Hall congruences give
\[
q\equiv43,67\pmod{120}.
\]
For the two least possible values,
\[
43^2-43+1=1807,
\qquad
67^2-67+1=4423,
\]
and neither integer divides $5760$.  For every larger admissible $q$ we have $q\ge163$, and hence
\[
D_r=q^2-q+1>5760.
\]
Thus $D_r\nmid|H|$ in every case.  Lemma~\ref{lem:rank3-orbit-div} gives $\alpha^H\ne0$.  Hence the even-degree character $\alpha$ occurs in $1_H^G$, completing the unitary case.
\end{proof}

\subsection{The case $2\in\pi$ and $3,p\notin\pi$}

Put
\[
\tau=(\pi\cap\pi(G))\setminus\{2\}.
\]
By \cite[Theorem~5.2]{VRsurvey}, under these hypotheses a Hall $\pi$-subgroup $H$ is contained in the normalizer of a maximal torus $T$ whose normalizer contains a Sylow $2$-subgroup of $G$; moreover, $H$ contains a normal abelian Hall $\tau$-subgroup contained in $T$.  The maximal tori whose normalizers contain Sylow $2$-subgroups are described in the proof of \cite[Lemma~3.10]{VRsurvey}.  For the classical groups this description gives the one- and two-dimensional decompositions, with the residual line or orthogonal summand when one occurs, used below.  Only the containment of $H$ in this decomposition stabilizer is used below; the full torus normalizer need not lie in it.

\begin{proposition}\label{prop:even-threefree-large}
Let $G$ be a simple classical group in characteristic $p$, let $H\in\Hall_\pi(G)$, and assume
\[
2\in\pi,\qquad 3,p\notin\pi.
\]
Then $1_H^G$ has an irreducible constituent of even degree for
\[
\PSp_{2n}(q),\qquad \PSL_n(q)\ (n\ge3),\qquad \PSU_n(q)\ (n\ge4),
\]
and for all orthogonal torus-normalizer families arising from the decomposition rows of \cite[Lemma~3.10]{VRsurvey} and \cite[Lemma~6.7(d),(e)]{RV10}.
\end{proposition}

\begin{proof}
For the symplectic groups, the proof of \cite[Lemma~3.10]{VRsurvey} places $H$ in the decomposition subgroup
\[
\Sp_2(q)\wr S_n.
\]
Proposition~\ref{prop:even-sp} shows that the relevant nonprincipal constituent of the projective-point permutation character has a nonzero fixed vector for this decomposition subgroup.  Since $H$ is contained in that subgroup, the constituent also has a nonzero $H$-fixed vector.

For orthogonal groups, the proof of \cite[Lemma~3.10]{VRsurvey} places $H$ in the decomposition stabilizer $M$ considered in Proposition~\ref{prop:even-orth-large}: the repeated nondegenerate $2$-spaces have the type determined by $q\pmod4$, and the residual orthogonal summand is retained when it occurs.  The two support types used in Proposition~\ref{prop:even-orth-large} are available in this decomposition.  In the hyperbolic case, and in the anisotropic case with at least four repeated blocks, both test points have zero residual component.  With three anisotropic blocks, the second test point uses the residual summand, covering both a nonsingular $1$-space and a residual nondegenerate $2$-space.  In the two mixed rows of \cite[Lemma~6.7(d),(e)]{RV10}, four repeated $2$-spaces suffice and both test points may be chosen with zero residual component.  Proposition~\ref{prop:even-orth-large} therefore gives a nonprincipal constituent with a nonzero $M$-fixed vector, and hence with a nonzero $H$-fixed vector.

For $\PSL_n(q)$, the proof of \cite[Lemma~3.10]{VRsurvey} places $H$ in a one-dimensional monomial stabilizer when $e(2,q)=1$, and in the corresponding two-dimensional block stabilizer, with at most one residual line, when $e(2,q)=2$.  If $n$ is odd, use the deleted point constituent $\gamma=\chi^{(n-1,1)}$.  The decomposition stabilizer has at least two orbits on projective points, represented by a point in one block and a point with support in at least two blocks.  Hence the projective-point permutation character has at least two fixed vectors on restriction to the stabilizer.  Since its only nonprincipal constituent is $\gamma$, we have $\gamma^M\ne0$.  Here
\[
\gamma(1)=q \cdot \frac{q^{n-1}-1}{q-1}=q(1+q+\cdots+q^{n-2}).
\]
As $q$ and $n$ are odd, the sum has $n-1$ terms, so $\gamma(1)$ is even.  If $4\mid n$, use $\delta=\chi^{(n-2,2)}$, and if $n\equiv2\pmod4$, use $\epsilon=\chi^{(n-2,1,1)}$.  Their degrees are even in the indicated parity classes, and Lemma~\ref{lem:linear-decomp-fixed} gives nonzero fixed vectors for the corresponding decomposition stabilizer, hence for $H$.

For $\PSU_n(q)$, let $M$ be the image in $G=\PSU_n(q)$ of the decomposition stabilizer containing $H$ in the proof of \cite[Lemma~3.10]{VRsurvey}.  If $e(2,q)=1$, then $M$ stabilizes a decomposition into repeated nondegenerate $1$-spaces, together with a residual line when $n$ is odd.  Lemma~\ref{lem:unitary-rank3} gives nonzero $M$-fixed vectors in both nonprincipal constituents of the isotropic-point rank-three action, except when $(q,n)=(3,5)$, where it gives an even-degree constituent with a nonzero $M$-fixed vector.  Outside this exceptional pair, exactly one of the two nonprincipal constituents has even degree.

If $e(2,q)=2$, then $M$ stabilizes a decomposition into repeated nondegenerate $2$-spaces, again with a residual line when $n$ is odd.  For $m\ge3$, the two orbit types used in Proposition~\ref{prop:even-A-large} have adjacency values $(m-2)(q+1)$ and $(m-3)(q+1)+1$; for $m=2$ the values are $0$ and $2$.  They are distinct, so the indicator of the one-block isotropic-point orbit has nonzero projection onto both nonprincipal eigenspaces.  Thus both nonprincipal constituents have nonzero $M$-fixed vectors.  In either case there is an even-degree constituent with a nonzero $M$-fixed vector.  Since $H\le M\le G$, Lemma~\ref{lem:overgroup} shows that this constituent occurs in $1_H^G$.
\end{proof}

\subsection{Rank-three orbit divisibility}

We need the following orbit-divisibility lemma for the exceptional orthogonal cases.

\begin{lemma}\label{lem:rank3-orbit-div}
Let a finite group $G$ act transitively with rank three on a finite set $X$, and let $A$ be the adjacency matrix of one of the two nontrivial orbital graphs.  Since the action is transitive of rank three, the permutation module is multiplicity free.  Write
\[
\mathbb C X=1_G\oplus V_r\oplus V_s,
\]
where $V_r$ and $V_s$ are the two nonprincipal irreducible eigenspaces and $A$ has eigenvalues $k,r,s$ on $1_G,V_r,V_s$, respectively.  Let $H\le G$, and put
\[
D_r=\frac{|X|}{\gcd(|X|,k-s)},
\qquad
D_s=\frac{|X|}{\gcd(|X|,k-r)}.
\]
Then
\[
D_r\nmid |H|\quad\Longrightarrow\quad V_r^H\ne0,
\]
and
\[
D_s\nmid |H|\quad\Longrightarrow\quad V_s^H\ne0.
\]
\end{lemma}

\begin{proof}
Let $O_1,\ldots,O_t$ be the $H$-orbits on $X$, with $n_i=|O_i|$, and let $Q=(q_{ij})$ be the quotient matrix of the orbit partition for $A$.  Since the orbits of a group of automorphisms form an equitable partition, $Q$ represents the action of $A$ on $(\mathbb C X)^H$.

Assume first that $V_r^H=0$.  If $t=1$, then $H$ is transitive on $X$, so $|X|$ divides $|H|$, and hence $D_r$ divides $|H|$.  We may therefore assume $t\ge2$.  The spectrum of $Q$ is then $k$ together with $s$ of multiplicity $t-1$.  The vector ${\bf1}$ is a right $k$-eigenvector.  On the space of orbit-constant vectors, use the weighted inner product
\[
\langle x,y\rangle_H=\sum_{i=1}^t n_i x_i\overline{y_i}.
\]
Since
\[
n_iq_{ij}=n_jq_{ji},
\]
the matrix $Q$ is self-adjoint for this inner product.  Hence the orthogonal complement of ${\bf1}$ is
\[
\left\{x=(x_1,\ldots,x_t):\sum_{i=1}^t n_i x_i=0\right\}.
\]
Under the assumption $V_r^H=0$, this whole complement is the $s$-eigenspace of $Q$.  Thus $Q$ acts as $k$ on $\langle{\bf1}\rangle$ and as $s$ on its orthogonal complement.  The orthogonal projection onto $\langle{\bf1}\rangle$ is
\[
x\longmapsto \frac{\sum_i n_i x_i}{|X|}{\bf1},
\]
and therefore
\begin{equation}\label{eq:rank3-quotient}
Q=sI+\frac{k-s}{|X|}{\bf1}(n_1,\ldots,n_t).
\end{equation}
For $i\ne j$ this gives
\[
q_{ij}=\frac{k-s}{|X|}n_j.
\]
Since $q_{ij}$ is an integer, $D_r$ divides $n_j$ for every $j$.  Each orbit length $n_j$ divides $|H|$, and hence $D_r$ divides $|H|$.  This proves the first implication by contraposition.  Interchanging $r$ and $s$ gives the second.
\end{proof}

The relevant rank-three graphs have the following parameters.

\begin{lemma}\label{lem:orth-polar-parameters}
Let $q$ be odd.
\begin{enumerate}
\item[(i)] For the collinearity graph on the singular $1$-spaces of a $7$-dimensional nondegenerate orthogonal space,
\[
v=(q+1)(q^2-q+1)(q^2+q+1),\qquad
k=q(q+1)(q^2+1),
\]
and the two restricted eigenvalues are $q^2-1$ and $-(q^2+1)$.
\item[(ii)] For the collinearity graph on the singular $1$-spaces of an $8$-dimensional plus-type orthogonal space,
\[
v=(q+1)^2(q^2+1)(q^2-q+1),\qquad
k=q(q^2+1)(q^2+q+1),
\]
and the two restricted eigenvalues are $q^3-1$ and $-(q^2+1)$.
\item[(iii)] For the collinearity graph on the singular $1$-spaces of a $9$-dimensional nondegenerate orthogonal space,
\[
v=(q+1)(q^2+1)(q^4+1),\qquad
k=q(q+1)(q^2-q+1)(q^2+q+1),
\]
and the two restricted eigenvalues are $q^3-1$ and $-(q^3+1)$.
\end{enumerate}
\end{lemma}

\begin{proof}
A nondegenerate orthogonal space of dimension $2m+1$ has
\[
\frac{q^{2m}-1}{q-1}
\]
singular $1$-spaces.  A plus-type space of dimension $2m$ has
\[
\frac{(q^{m-1}+1)(q^m-1)}{q-1}
\]
singular $1$-spaces.  These formulas are obtained by counting the zeros of the standard hyperbolic-coordinate quadratic forms.

First let $V$ have dimension $7$.  The point count gives
\[
v=\frac{q^6-1}{q-1}=(q+1)(q^2-q+1)(q^2+q+1).
\]
For a singular point $x$, the quotient $x^\perp/x$ is a $5$-dimensional nondegenerate orthogonal space.  Each singular point of this quotient has exactly $q$ singular lifts different from $x$.  Hence
\[
k=q\cdot\frac{q^4-1}{q-1}=q(q+1)(q^2+1).
\]
If $x$ and $y$ are nonadjacent, then $\langle x,y\rangle$ is a hyperbolic plane and its orthogonal complement is $5$-dimensional nondegenerate.  Thus the number of common neighbors is
\[
\mu=(q+1)(q^2+1).
\]
If $x$ and $y$ are adjacent and $L=\langle x,y\rangle$, then $L$ is the radical of $L^\perp$, while $L^\perp/L$ is $3$-dimensional nondegenerate.  It has $q+1$ singular points; over each there are $q^2$ singular points of $L^\perp$ outside $L$.  After deleting $x$ and $y$ we obtain
\[
\lambda=(q+1)+q^2(q+1)-2=q^3+q^2+q-1.
\]
The restricted eigenvalues are the roots of
\[
z^2-(\lambda-\mu)z-(k-\mu)=0,
\]
namely $q^2-1$ and $-(q^2+1)$.

Now let $V$ have dimension $8$ and plus type.  The same point count gives
\[
v=\frac{(q^3+1)(q^4-1)}{q-1}
=(q+1)^2(q^2+1)(q^2-q+1).
\]
Here $x^\perp/x$ is $6$-dimensional of plus type, so
\[
k=q(q^2+1)(q^2+q+1).
\]
For nonadjacent $x,y$, the orthogonal complement of the hyperbolic plane $\langle x,y\rangle$ is $6$-dimensional of plus type, and therefore
\[
\mu=(q^2+1)(q^2+q+1).
\]
For adjacent $x,y$ and $L=\langle x,y\rangle$, the quotient $L^\perp/L$ is $4$-dimensional of plus type and has $(q+1)^2$ singular points.  Consequently
\[
\lambda=(q+1)+q^2(q+1)^2-2
=q^4+2q^3+q^2+q-1.
\]
The same quadratic equation now has roots $q^3-1$ and $-(q^2+1)$.

Finally let $V$ have dimension $9$.  Then
\[
v=\frac{q^8-1}{q-1}=(q+1)(q^2+1)(q^4+1),
\]
and $x^\perp/x$ is $7$-dimensional, so
\[
k=q\cdot\frac{q^6-1}{q-1}
=q(q+1)(q^2-q+1)(q^2+q+1).
\]
If $x$ and $y$ are nonadjacent, the orthogonal complement of $\langle x,y\rangle$ is $7$-dimensional, and hence
\[
\mu=\frac{q^6-1}{q-1}.
\]
If $x$ and $y$ are adjacent and $L=\langle x,y\rangle$, then $L^\perp/L$ is $5$-dimensional.  Consequently
\[
\lambda=(q+1)+q^2\frac{q^4-1}{q-1}-2.
\]
The quadratic equation above now has roots $q^3-1$ and $-(q^3+1)$.
\end{proof}

\subsection{The exceptional groups $\Omega_9(q)$}

\begin{proposition}\label{prop:orth-nine}
Let $G=\Omega_9(q)$ and let $H$ be the Hall subgroup in \cite[Lemma~6.7(h)]{RV10}, so that
\[
H\cong2.\Omega_8^+(2).2.
\]
Then $1_H^G$ has an even-degree irreducible constituent.
\end{proposition}

\begin{proof}
Let $X$ be the singular $1$-spaces of the natural $9$-dimensional orthogonal module.  This is one of the classical rank-three actions in \cite[Theorem~1.1]{KL82}.  Write
\[
1_X=1+\alpha+\beta,
\]
where $\alpha$ is afforded by the eigenspace for $r=q^3-1$.  By Lemma~\ref{lem:orth-polar-parameters}(iii),
\[
|X|=(q+1)(q^2+1)(q^4+1),
\]
\[
k=q(q+1)(q^2-q+1)(q^2+q+1),
\qquad
s=-(q^3+1).
\]
Using
\[
1+\alpha(1)+\beta(1)=|X|
\]
and the trace identity
\[
k+r\alpha(1)+s\beta(1)=0,
\]
we obtain
\[
\alpha(1)=\frac{q(q+1)^2(q^2+1)(q^2-q+1)}2.
\]
This is even because $q$ is odd.

We apply Lemma~\ref{lem:rank3-orbit-div}.  Since
\[
k-s=(q+1)^2(q^2+1)(q^2-q+1),
\]
and
\[
\gcd(q^4+1,q+1)=2,
\qquad
\gcd(q^4+1,q^2-q+1)=1,
\]
we have
\[
D_r=\frac{|X|}{\gcd(|X|,k-s)}=\frac{q^4+1}{2}.
\]
The second gcd follows because, if a prime divides both factors, then $q^2\equiv q-1$, whence $q^4+1\equiv1-q$; it therefore divides $q-1$ and then also $1$.

Since $q$ is odd, $q^4+1\equiv2\pmod{16}$, so $D_r$ is odd.  Let $\ell$ be a prime divisor of $D_r$.  Then
\[
q^4\equiv-1\pmod\ell,
\]
so the multiplicative order of $q$ modulo $\ell$ is $8$.  Hence $\ell\equiv1\pmod8$.  On the other hand, every prime divisor of $H\cong2.\Omega_8^+(2).2$ belongs to $\{2,3,5,7\}$.  Thus $D_r\nmid|H|$.  Lemma~\ref{lem:rank3-orbit-div} gives $\alpha^H\ne0$, and therefore $\alpha$ is an even-degree constituent of $1_H^G$.
\end{proof}

\subsection{The exceptional groups $\Omega_7(q)$}

\begin{proposition}\label{prop:orth-seven}
Let $G=\Omega_7(q)$ and let $H$ be a Hall subgroup in \cite[Lemma~6.7(f)]{RV10}.  Thus
\[
H\cong\Omega_7(2),\qquad
\pi\cap\pi(G)=\{2,3,5,7\},\qquad
|H|=2^9\cdot3^4\cdot5\cdot7.
\]
Then $1_H^G$ has an even-degree irreducible constituent.
\end{proposition}

\begin{proof}
Let $X$ be the singular $1$-spaces of the natural $7$-dimensional orthogonal module.  This is one of the classical rank-three actions in \cite[Theorem~1.1]{KL82}.  The permutation character is
\[
1_X=1+\alpha+\beta,
\]
and for the collinearity graph
\[
|X|=(q+1)(q^2-q+1)(q^2+q+1),
\]
\[
k=q(q+1)(q^2+1),\qquad r=q^2-1,\qquad s=-(q^2+1).
\]
The $s$-eigenspace affords $\beta$, with
\[
\beta(1)=\frac{q(q+1)^2(q^2-q+1)}2,
\]
which is even for odd $q$.

By Lemma~\ref{lem:rank3-orbit-div}, it is enough to show that
\[
D_s=\frac{|X|}{\gcd(|X|,k-r)}
\]
does not divide $|H|$.  Since
\[
k-r=(q+1)^2(q^2-q+1),
\]
we have
\[
\gcd(|X|,k-r)
=(q+1)(q^2-q+1)\gcd(q^2+q+1,q+1).
\]
Now $q^2+q+1\equiv1\pmod{q+1}$, so
\[
D_s=q^2+q+1.
\]

The Hall conditions in \cite[Lemma~6.7(f)]{RV10} imply
\[
q\equiv\pm2\pmod5.
\]
Hence $5\nmid D_s$.  Also $D_s$ is odd.  If $D_s$ divided $|H|$, it would therefore divide
\[
3^4\cdot7=567.
\]
The Hall conditions imply $q\ge173$.  Indeed,
\[
 |G|=\frac12 q^9(q^2-1)(q^4-1)(q^6-1),
\]
and the $\{2,3,5,7\}$-part of $|G|$ must be exactly
$2^9\cdot3^4\cdot5\cdot7$.  We show that no admissible $q<173$ exists.  For an odd prime power $q<173$ with $q\equiv\pm2\pmod5$, the order formula gives
\[
v_2(|G|)=9,\qquad v_3(|G|)=4.
\]
Factoring $q^2-1$, $q^4-1$, and $q^6-1$ and applying LTE converts these equalities into congruence conditions modulo $2^6\cdot3^3\cdot5=8640$.  Among the odd prime powers $q<173$ satisfying $q\equiv\pm2\pmod5$, the surviving residue classes are represented exactly by
\[
q=13,43,67,83,157.
\]
For these values, the order formula gives $v_7(|G|)=3$ for $q=13,43,83$, $v_7(|G|)=2$ for $q=67$, and $v_5(|G|)=2$ for $q=157$.  None is admissible.  Thus $q\ge173$, and hence
\[
D_s=q^2+q+1>567,
\]
a contradiction.  Therefore $\beta^H\ne0$, and the even-degree character $\beta$ occurs in $1_H^G$.
\end{proof}

\subsection{The groups $\PSU_3(q)$}

Consider the torus-normalizer case for $\PSU_3(q)$.  The natural action on isotropic points has an even-degree irreducible constituent with nonzero $H$-fixed vectors.

\begin{proposition}\label{prop:U3}
Let $G=\PSU_3(q)$ with $q$ odd, and let $H$ be a Hall subgroup in the case $2\in\pi$, $3,p\notin\pi$.  Then $1_H^G$ has an even-degree irreducible constituent.
\end{proposition}

\begin{proof}
Let $X$ be the set of isotropic $1$-spaces of the natural unitary module.  Then
\[
|X|=q^3+1,
\]
and $G$ is doubly transitive on $X$.  Fix $\infty\in X$ and write
\[
B=G_\infty=U\rtimes T,
\qquad |U|=q^3,
\qquad |T|=\frac{q^2-1}{d},
\qquad d=(3,q+1),
\]
where $T$ is cyclic.  The group $U$ acts regularly on $X\setminus\{\infty\}$.  If $N$ is the stabilizer of an unordered pair of distinct points, then orbit--stabilizer gives
\[
|N|=2\cdot\frac{q^2-1}{d}=2|T|.
\]
Thus $N=N_G(T)$.  By \cite[Theorem~5.2]{VRsurvey}, after conjugating $H$ if necessary we may assume that $H\le N_G(T)=N$.

Before passing to the projective quotient, write $U$ in its standard upper-unitriangular form.  Its elements may be written
\[
u(a,b)=
\begin{pmatrix}
1&a&b\\
0&1&-a^q\\
0&0&1
\end{pmatrix},
\qquad
b+b^q=-a^{q+1},
\]
and a diagonal torus element may be represented by
\[
h_\lambda=\operatorname{diag}(\lambda,\lambda^{q-1},\lambda^{-q}),
\qquad \lambda\in\mathbb F_{q^2}^{\times}.
\]
Conjugation sends the $a$-coordinate to $\lambda^{q-2}a$.  The kernel of this action has order
\[
\gcd(q^2-1,q-2)=d,
\]
which is exactly the scalar kernel.  Thus $T$ acts faithfully on $U/Z(U)$.

The $U$-orbits on unordered pairs not containing $\infty$ are represented by
\[
\{1,u\},\qquad 1\ne u\in U,
\]
with $u$ and $u^{-1}$ representing the same orbit.  Choose $u=u(a,b)$ with $a\ne0$.  If $t\in T$ stabilizes this $U$-orbit, then $u^t=u$ or $u^{-1}$.  The first possibility gives $t=1$.  In the second, faithfulness on $U/Z(U)$ forces $t$ to be the unique involution of $T$, hence it sends $a$ to $-a$.  A representative $h_\lambda$ of this involution satisfies $\lambda^{q-2}=-1$.  Since its order divides both $2(q-2)$ and $q^2-1$, it divides $2d$, and therefore $\lambda^{q+1}=1$.  Thus this involution sends
\[
u(a,b)\longmapsto u(-a,b).
\]
On the other hand
\[
u(a,b)^{-1}=u(-a,b^q).
\]
For fixed $a\ne0$, the equation $b+b^q=-a^{q+1}$ has exactly $q$ solutions, and exactly one lies in $\mathbb F_q$.  Since $q$ is odd, choose $b\notin\mathbb F_q$.  Then the corresponding $U$-orbit has trivial stabilizer in $T$.

It follows that the $U$-fixed subspace of the unordered-pair permutation module contains the regular $T$-module.  Hence every linear character $\theta\in\Irr(T)$, inflated across $U$, occurs in the restriction of $1_N^G$ to $B$.  Choose $\theta$ faithful.

The induced character is irreducible.  Since the action on $X$ is doubly transitive, $B\backslash G/B$ has two double cosets, represented by $1$ and a Weyl element $w$, and $B\cap B^w=T$.  In the above torus parametrization $w$ acts by
\[
t\longmapsto t^{-q}.
\]
If $\theta^w=\theta$, faithfulness would imply
\[
-q\equiv1\pmod{|T|}.
\]
This is impossible because $|T|=(q^2-1)/d$ does not divide $q+1$ for odd $q\ge3$.  Mackey's formula therefore gives
\[
\left[\theta^G,\theta^G\right]_G
=1+[\theta,\theta^w]_T
=1.
\]
Thus
\[
\chi:=\theta^G\in\Irr(G),
\qquad
\chi(1)=[G:B]=q^3+1.
\]
Frobenius reciprocity and the regular-orbit argument give
\[
[1_N^G,\chi]>0.
\]
Since $q$ is odd, $q^3+1$ is even.  Finally $H\le N$, so Lemma~\ref{lem:overgroup} gives $\chi\le1_H^G$.
\end{proof}

\subsection{The exceptional groups $P\Omega_8^+(q)$}

It remains to consider row (g) of \cite[Lemma~6.7]{RV10}.

\begin{proposition}\label{prop:orth-eight}
Let $G=P\Omega_8^+(q)$ be simple, and let $H$ be the image in $G$ of a Hall subgroup in \cite[Lemma~6.7(g)]{RV10}.  Thus
\[
H\cong\Omega_8^+(2),
\qquad
\pi\cap\pi(G)=\{2,3,5,7\}.
\]
Then $1_H^G$ has an even-degree irreducible constituent.
\end{proposition}

\begin{proof}
Let $X$ be the singular $1$-spaces of the natural $8$-dimensional plus-type module.  This is one of the classical rank-three actions in \cite[Theorem~1.1]{KL82}.  The corresponding permutation character has the form
\[
1_X=1+\alpha+\beta.
\]
For the collinearity graph one has
\[
|X|=(q+1)^2(q^2+1)(q^2-q+1),
\]
\[
k=q(q^2+1)(q^2+q+1),
\qquad
r=q^3-1,
\qquad
s=-(q^2+1).
\]
The $r$-eigenspace affords $\alpha$, with
\[
\alpha(1)=q(q^2+1)^2,
\]
which is even because $q$ is odd.

Revin--Vdovin formulate row (g) in $\Omega_8^+(q)$, where the Hall subgroup has structure $2.\Omega_8^+(2)$.  In this embedding the natural $8$-dimensional module restricts irreducibly to $2.\Omega_8^+(2)$.  Its central involution therefore acts as a scalar by Schur's lemma; since the characteristic is odd and the involution is nontrivial, that scalar is $-1$.  Thus it is the scalar central involution of $\Omega_8^+(q)$, and after passage to the simple quotient $G=P\Omega_8^+(q)$ its image is
\[
H\cong\Omega_8^+(2),
\qquad
|H|=2^{12}\cdot3^5\cdot5^2\cdot7.
\]
The action on singular $1$-spaces factors through this quotient.  It is therefore enough to show that $\alpha^H\ne0$.

We apply Lemma~\ref{lem:rank3-orbit-div}.  Since
\[
k-s=(q+1)(q^2+1)^2,
\]
and, for odd $q$,
\[
\gcd\bigl(q^2+1,(q+1)(q^2-q+1)\bigr)=2,
\]
we obtain
\begin{equation}\label{eq:orth8-divisor}
D_r=\frac{|X|}{\gcd(|X|,k-s)}
=\frac{(q+1)(q^2-q+1)}2.
\end{equation}
Hence $\alpha^H=0$ would force $D_r$ to divide $|H|$.

Suppose first that $q\ge559$.  Let $t$ be the number of $H$-orbits on $X$, and write
\[
a=\dim\alpha^H,\qquad b=\dim\beta^H.
\]
Then $t=1+a+b$.  Let $A$ be the adjacency matrix and let
\[
P_H=\frac1{|H|}\sum_{g\in H}g
\]
be the projection onto the $H$-fixed subspace.  Put $u=\operatorname{tr}(AP_H)$.  Spectrally,
\[
u=k+ra+sb.
\]
On the other hand
\[
u=\frac1{|H|}\sum_{g\in H}\operatorname{tr}(Ag)\ge0,
\]
because $\operatorname{tr}(Ag)$ counts the points $x\in X$ for which $x$ is adjacent to $g(x)$.
If $a=0$, then $b=t-1$ and therefore
\[
0\le u=k+s(t-1),
\qquad
(q^2+1)(t-1)\le k.
\]
On the other hand $t\ge |X|/|H|$.  Hence $a=0$ would imply
\[
(q^2+1)\left(\frac{|X|}{|H|}-1\right)\le k.
\]
After cancellation this is equivalent to
\[
q^3\le |H|-1=174182399.
\]
Since $559^3>174182399$, every admissible $q\ge559$ has $\alpha^H\ne0$.

Now suppose that $q<559$ is admissible.  For such an odd prime power $q$, the order formula for $P\Omega_8^+(q)$ and the equality
\[
|H|=2^{12}\cdot3^5\cdot5^2\cdot7
\]
together with the congruence conditions in \cite[Lemma~6.7(g)]{RV10} restrict $q$ further.  Applying LTE to $q^2-1$, $q^4-1$, and $q^6-1$ converts the required $2$- and $3$-parts into congruence conditions on $q$.  Combining these congruences with those in \cite[Lemma~6.7(g)]{RV10}, the odd prime powers below $559$ that remain are exactly
\[
q=173,277,283,317,347.
\]
For these five values the integer in \eqref{eq:orth8-divisor} factors as follows:
\[
\begin{array}{c|l}
q&D_r\\ \hline
173&3^2\cdot7\cdot13\cdot29\cdot109,\\
277&13\cdot139\cdot5881,\\
283&2\cdot7\cdot13\cdot71\cdot877,\\
317&3^2\cdot53\cdot33391,\\
347&2\cdot3^2\cdot29\cdot31\cdot1291.
\end{array}
\]
In every case $D_r$ has a prime divisor outside $\{2,3,5,7\}$, and therefore $D_r\nmid|H|$.  Lemma~\ref{lem:rank3-orbit-div} gives $\alpha^H\ne0$.  Thus $\alpha$ occurs in $1_H^G$, and its degree is even.
\end{proof}

\begin{theorem}\label{thm:classical-all}
Let $G$ be a finite simple classical group and let $1<H<G$ with $H\in\Hall_\pi(G)$.  Then there exists
\[
\chi\in\Irr(1_H^G)
\]
such that $\chi(1)$ is divisible by a prime in $\pi$.
\end{theorem}

\begin{proof}
The case $2\notin\pi$ is Theorem~\ref{thm:classical}.  Assume $2\in\pi$.  If the defining characteristic belongs to $\pi$, Proposition~\ref{prop:defining-even} applies.  In nondefining characteristic, Proposition~\ref{prop:even-A1} treats the rank-one linear groups, while Propositions~\ref{prop:even-sp}, \ref{prop:even-orth-large}, \ref{prop:even-A-large}, and \ref{prop:even-threefree-large} treat the decomposition and torus-normalizer families.  The exceptional embeddings are covered by Proposition~\ref{prop:even-A6} in type $A$, Proposition~\ref{prop:U3} for $\PSU_3(q)$, and Propositions~\ref{prop:orth-seven}, \ref{prop:orth-eight}, and \ref{prop:orth-nine} for the orthogonal groups.
\end{proof}

Liu--Yang--Zhang reduce their normality criterion to the simple-group assertion; see \cite{LYZ}.  They treat the alternating, sporadic, and exceptional groups of Lie type, while Theorem~\ref{thm:classical-all} proves it for classical simple groups.  Together with their reduction, this proves the following theorem.

\begin{theorem}\label{thm:normality}
Let $G$ be a finite group and let $H\in\Hall_\pi(G)$.  Then $H$ is normal in $G$ if and only if every irreducible constituent $\chi$ of $1_H^G$ has $\pi'$-degree.
\end{theorem}

\section*{Acknowledgements}
The author was supported in part by a grant from the Simons Foundation (Grant No.~918096).

\section*{Disclosure Statement}
The author reports no conflict of interest.

\section*{Data Availability Statement}
No data were used in this work.

\end{document}